\documentclass[a4paper,10pt,reqno]{article}

\usepackage[utf8]{inputenc}
\usepackage[T1]{fontenc}
\usepackage{lmodern}
\usepackage[english]{babel}
\usepackage{microtype}
\usepackage{a4wide}

\usepackage{amsmath,amssymb,amsfonts,amsthm}
\usepackage{mathtools,accents}
\usepackage{mathrsfs}
\usepackage{aliascnt}
\usepackage{braket}
\usepackage{bm}
\usepackage{physics}

\usepackage[a4paper,margin=3cm]{geometry}
\usepackage[citecolor=blue,colorlinks]{hyperref}

\usepackage{enumerate}
\usepackage{xcolor}

\makeatletter
\g@addto@macro\@floatboxreset\centering
\makeatother

\makeatletter
\def\newaliasedtheorem#1[#2]#3{
  \newaliascnt{#1@alt}{#2}
  \newtheorem{#1}[#1@alt]{#3}
  \expandafter\newcommand\csname #1@altname\endcsname{#3}
}
\makeatother

\numberwithin{equation}{section}

\newtheoremstyle{slanted}{\topsep}{\topsep}{\slshape}{}{\bfseries}{.}{.5em}{}

\theoremstyle{plain}
\newtheorem{theorem}{Theorem}[section]
\newaliasedtheorem{proposition}[theorem]{Proposition}
\newaliasedtheorem{lemma}[theorem]{Lemma}
\newaliasedtheorem{corollary}[theorem]{Corollary}
\newaliasedtheorem{conjecture}[theorem]{Conjecture}
\newaliasedtheorem{counterexample}[theorem]{Counterexample}

\theoremstyle{definition}
\newaliasedtheorem{definition}[theorem]{Definition}
\newaliasedtheorem{question}[theorem]{Question}
\newaliasedtheorem{example}[theorem]{Example}

\theoremstyle{remark}
\newaliasedtheorem{remark}[theorem]{Remark}

\newcommand{\setN}{\mathbb{N}}

\newcommand{\setR}{\mathbb{R}}
\renewcommand{\subset}{\subseteq}

\newcommand{\defeq}{\mathrel{\mathop:}=}

\newcommand{\mres}{\mathbin{\vrule height 1.6ex depth 0pt width 0.13ex\vrule height 0.13ex depth 0pt width 1.3ex}}

\let\altphi\phi
\let\phi\varphi
\let\varphi\altphi
\let\altphi\undefined

\newcommand{\weakto}{\rightharpoonup}

\let\div\undefined
\DeclareMathOperator{\div}{div}

\newcommand{\di}{\mathrm{d}}

\newcommand{\loc}{{ loc}}
\newcommand{\res}{\mathop{\hbox{\vrule height 7pt width .5pt depth 0pt
\vrule height .5pt width 6pt depth 0pt}}\nolimits}

\DeclareMathOperator{\supp}{supp}

\newcommand{\Ch}{{\sf Ch}}

\DeclareMathOperator{\Lip}{Lip}
\DeclareMathOperator{\lip}{lip}

\newcommand{\Lipbs}{\Lip_{bs}}

\newcommand{\haus}{\mathscr{H}}

\newcommand{\XX}{{\mathsf{X}}}
\newcommand{\YY}{{\mathsf{Y}}}
\newcommand{\ZZ}{{\mathsf{Z}}}

\newcommand{\dist}{\mathsf{d}}

\newcommand{\meas}{\mathfrak{m}}
\newcommand{\mass}{\mathfrak{m}}

\newcommand{\heat}{{\mathrm {h}}}
\newcommand{\hess}{{\mathrm{Hess}}}
\newcommand{\tra}{{\mathrm{tr}}}

\DeclareMathOperator{\RCD}{RCD}
\DeclareMathOperator{\ncRCD}{ncRCD}

\DeclareMathOperator{\CD}{CD}

\newfont{\tmpf}{cmsy10 scaled 2500}

\newcommand{\X}{\XX}  							
\newcommand{\sfd}{\dist}  							
\newcommand{\mm}{\meas}							
\newcommand{\R}{\setR}							
\newcommand{\N}{\setN}							
\newcommand{\D}{{\rm D}}
\renewcommand{\d}{\di}							
\newcommand{\Y}{{\YY}}
\DeclareMathOperator*{\esssup}{{\rm ess\,sup}}
\newcommand{\fr}{\penalty-20\null\hfill$\blacksquare$}         

\begin{document}

\title{Weakly non-collapsed RCD spaces are strongly  non-collapsed}
\author{
Camillo Brena
\thanks{Scuola Normale Superiore, \url{camillo.brena@sns.it}}
  \and
Nicola Gigli
\thanks{Scuola Internazionale Superiore di Studi Avanzati, \url{ngigli@sissa.it}}
  \and
Shouhei Honda
\thanks{Tohoku University, \url{shouhei.honda.e4@tohoku.ac.jp}}
  \and
Xingyu Zhu
\thanks{Georgia Institute of Technology, \url{xyzhu@gatech.edu}} } 
\maketitle

\begin{abstract}
We prove that any weakly non-collapsed RCD space is actually non-collapsed, up to a renormalization of the measure. This confirms a conjecture raised by De Philippis and the second named author in full generality.

One of the auxiliary results of independent interest that we obtain is about the link between the properties
\begin{itemize}
\item[-] $\mathrm{tr}(\mathrm{Hess}f)=\Delta f$ on $U\subset\X$ for every $f$ sufficiently regular,
\item[-] $\mm=c\haus^n$ on $U\subset\X$ for some $c>0$,
\end{itemize} 
where  $U\subset\X$ is open and $\X$ is a - possibly collapsed - RCD space of essential dimension  $n$.
\end{abstract}

\tableofcontents


\section{Introduction}
\subsection{Main result and some comments}
The modern theory of metric measure spaces with Ricci curvature bounded from below began with the seminal papers \cite{Lott-Villani09}, \cite{Sturm06I,Sturm06II}, where lower bounds on the Ricci curvature were imposed via suitable convexity properties of entropy functionals in the geometry of optimal transportation. It turned out that the resulting class $\CD(K,N)$ of spaces contains smooth Finslerian structures, a property that - at least for some geometric applications - is undesirable. This was one of the motivations that led the second named author develop a research program about heat flow on $\CD$ spaces (see \cite{Gigli10}, \cite{Gigli-Kuwada-Ohta10}, \cite{AmbrosioGigliSavare11},  \cite{Ambrosio_2014}) that ultimately led in \cite{Gigli12} to the definition of $\RCD$ spaces as those $\CD$ spaces for which the Sobolev space $W^{1,2}$ is Hilbert. We refer to \cite{Villani2017}, \cite{AmbICM} for an account of the theory and more detailed bibliography.

The study of (${\sf R}$)$\CD$ spaces has been strongly influenced by the research program carried out in the late nineties by Cheeger-Colding (see \cite{Cheeger-Colding96,Cheeger-Colding97I,Cheeger-Colding97II,Cheeger-Colding97III}) about the structure of Ricci-limit spaces, i.e.\ those spaces arising as Gromov-Hausdorff limits of smooth Riemannian manifolds with constant dimension  and a uniform lower bound on the Ricci curvature. One of the things that emerged by their analysis (strictly related to Colding's volume convergence theorem \cite{Colding97}) is that the dimension of the limit space is always bounded from above by the dimension of the manifolds, and that if equality holds, then the limit structure - that in this case is called \emph{non-collapsed Ricci limit space} - has better regularity properties.

It is therefore natural to look for a synthetic counterpart of this concept, whose definition should not rely on properties of approximating sequences, but rather be based on intrinsic properties of the space in consideration. Inspecting the properties of non-collapsed Ricci limit spaces, in \cite{DPG17} it has been proposed the following definition:
\begin{definition}
Let $K\in\R$ and $N\in[1,\infty)$. A space $(\X,\sfd,\mm)$ is called \emph{non-collapsed} $\RCD(K,N)$ space, $\ncRCD(K,N)$ in short, if it is an $\RCD(K,N)$ space and moreover $\mm=\haus^N$.
\end{definition}
Notice that from the structural properties of $\RCD$ spaces (\cite{Mondino-Naber14}, \cite{KelMon16}, \cite{GP16-2}) it is clear that if $(\X,\sfd,\mm)$ is $\ncRCD(K,N)$, then $N$ must be an integer. Also, subsequent analysis showed that, as expected, $\ncRCD$ spaces have better regularity properties than arbitrary $\RCD$ spaces (see for instance \cite{antonelli2019volume}).

\bigskip

The analysis carried out by Cheeger-Colding and the analogy with the study of the Bakry-\'Emery $N$-Ricci curvature tensor (see \eqref{eq:beric} and the subsequent discussion) suggest that in fact $\ncRCD(K,N)$ spaces should be identifiable among $\RCD(K,N)$ ones by properties seemingly weaker than the one $\mm=\haus^N$. To be more precise we need to introduce the $N$-dimensional (Bishop-Gromov) density $\theta_N[\X,\sfd,\mm]:\X\to[0,\infty]$ as
\[
\theta_N[\X,\sfd,\mm](x):=\lim_{r\to0^+}\frac{\mm(B_r(x))}{\omega_Nr^N}\qquad\forall x\in\X,
\]
where $\omega_N:=\frac{\pi^{N/2}}{\int_0^\infty t^{N/2}e^{-t}\,\d t}$ is, for $N\in\N$, the volume of the unit ball of $\R^N$ (notice that the existence of the limit is a consequence of the Bishop-Gromov inequality). It is worth pointing out that standard results about differentiation of measures ensure that if $\haus^N$ is a Radon measure on $\X$, then
\begin{equation}\notag
\limsup_{r\to0^+}\frac{\haus^N(B_r(x))}{\omega_Nr^N}\leq 1\qquad\text{for $\haus^N$-a.e.\ }x\in\X.
\end{equation}
In particular, if $\X$ is a $\ncRCD(K,N)$ space we have
\begin{equation}\label{eq:densityest1}
	\theta_N[\X,\sfd,\mm](x)\leq 1\qquad\forall x\in\XX.
\end{equation}

Then the following conjecture is raised in \cite{DPG17}:
\begin{conjecture}[De Philippis-Gigli]\label{conj}
If
\begin{equation}\label{eq:positive}
\meas \left( \left\{x \in \XX: \theta_N[\XX,\dist,\meas](x)<\infty\right\}\right)>0,
\end{equation}
then $\meas=c\haus^N$ for some $c \in (0, \infty)$. In particular $(\XX, \dist,c^{-1} \meas)$ is a  $\ncRCD(K, N)$ space.
\end{conjecture}
Let us make few comments about the statement of the conjecture and its validity.

First of all, we remark that condition \eqref{eq:positive} cannot be replaced by the weaker one
\[
\left\{x \in \XX: \theta_N[\XX,\dist,\meas](x)<\infty\right\} \neq \emptyset
\]
because for instance the metric measure space $([0, \pi], \dist_{\mathbb{R}}, \sin^{N-1}t\di t)$ is an $\RCD(N-1, N)$ space, the density $\theta_N$ is finite on $\{0, \pi\}$ which is null with respect to the reference measure $\sin^{N-1}t\di t$, and for any $N>1$, $\sin^{N-1}t\di t$ does not coincides with $c\haus^N$ for any $c \in (0, \infty)$.

Moreover let us point out that the Hausdorff dimension of any $\CD(K,N)$ space $\X$ is at most $N$ \cite{Sturm06II}. In this sense, the assumption in Conjecture \ref{conj} amounts at asking for the existence of a `big' portion of the space with maximal dimension (notice for instance that if $\mm\ll\haus^\alpha$ for some $\alpha<N$, then $\theta_N=+\infty$ $\mm$-a.e.). Such `maximality' of $N$ in the conjecture plays an important role. To see why, consider  an $n$-dimensional weighted Riemannian manifold $(M, g, e^{-V}\di \mathrm{Vol}_g)$, where $V\in C^\infty(M)$ and $\mathrm{Vol}_g$ denotes the Riemannian volume measure and recall that for $N\geq1$ the Bakry-\'Emery $N$-Ricci curvature tensor is defined as
\begin{equation}
\label{eq:beric}
\mathrm{Ric}_N:=\left\{\begin{array}{ll}
 \mathrm{Ric}_g+ \hess_g(V) -\frac{\dd V\otimes \dd V}{N-n}&\qquad\text{ if }N>n,\\
 \mathrm{Ric}_g&\qquad\text{ if $N=n$ and $V$ is constant},\\
 -\infty&\qquad\text{ otherwise},
\end{array}\right.
\end{equation}
where $\mathrm{Ric}_g$ is the standard Ricci curvature induced by the metric tensor $g$ and defined as trace of the Riemann curvature tensor. It is then known (see \cite{BakryEmery85},   \cite{AmbrosioGigliSavare12}, \cite{Erbar-Kuwada-Sturm13}) that  
\begin{equation}
\label{eq:bekn}
\text{$(M, \dist_g, e^{-V}\mathrm{Vol}_g)$ is an $\RCD(K, N)$ space if and only if $\mathrm{Ric}_N\geq Kg$.}
\end{equation}
On the other hand, it is clear that $(e^{-V} \mathrm{Vol}_g)(B_r(x))\sim r^n$ for every $x\in M$ as $r\rightarrow 0^+$, thus assumption \eqref{eq:positive} holds if and only if $n=N$, and this information together with $\mathrm{Ric}_N\geq Kg$ forces $V$ to be constant by the very definition of $\mathrm{Ric}_N$.

\bigskip

It is now time to point out that thanks to the main result of \cite{bru2018constancy} - and the aforementioned structural properties - we now know that any $\RCD(K,N)$ space $(\X,\sfd,\mm)$ admits an \emph{essential dimension} $n\in\N\cap [1,N]$, meaning in particular that $\mm\ll\haus^n\ll\mm$ on the Borel set ${\mathcal R}_n^*$ (see \eqref{eq:rnstar} below), where $\mass(\XX\setminus{\mathcal R}_n^* )=0$. We thus see from general results about differentiation of measures that
\begin{equation}
\label{eq:fromloctoglob}
\text{if \eqref{eq:positive} holds, then we have $\theta_N[\X,\sfd,\mm]<\infty\quad\mm$-a.e.}.
\end{equation}
$\RCD(K,N)$ spaces for which $\theta_N[\X,\sfd,\mm]$ is finite $\mm$-a.e.\ have been called \emph{weakly non-collapsed $\RCD$ spaces} in \cite{DPG17}, while spaces such that $\theta_N[\X,\sfd,\mm]$ is finite for \emph{every} point have been called `non-collapsed' in \cite{Kita17}. It is then clear from \eqref{eq:densityest1} that
\[
\text{non-collapsed} \Longrightarrow \text{non-collapsed in the sense of \cite{Kita17}} \Longrightarrow \text{weakly non-collapsed} 
\]
and from \eqref{eq:fromloctoglob} that proving Conjecture \ref{conj} is equivalent to proving that these three `non-collapsing conditions' are equivalent (up to multiplying the reference measure by a scalar).

It is known that the conjecture holds true in the following three cases:
\begin{enumerate}
\item $(\XX, \dist)$ has an upper bound on sectional curvature in a synthetic sense, namely, it is a CAT$(\kappa)$ space for some $\kappa>0$:  \cite{kapovitch2019weakly}
\item $(\XX, \dist)$ is isometric to a smooth Riemannian manifold, possibly with boundary: \cite{han2021measure}.
\item $(\XX, \dist)$ is compact: \cite{H19}.
\end{enumerate}
Our main result is the resolution of Conjecture \ref{conj} in full generality:
\begin{theorem}\label{thm:mainres}
Conjecture \ref{conj} holds true.
\end{theorem}
Notice that as a consequence of our main result, we obtain that if the Hausdorff dimension of an $\RCD(K,N)$ space is $N$, then also its topological dimension is $N$ (we refer to \cite{Nagata83} for the relevant definitions).
Indeed, under this assumption Theorem \ref{thm:equivnon} and our main result imply that the space is, up to a scalar multiple of the reference measure, a $\ncRCD(K,N)$ space. Then from   the Reifenberg flatness around a regular point (see \cite{Cheeger-Colding97I} and then \cite{DPG17}, \cite{KapMon19}) we see that   any regular point has a neighbourhood which is homeomorphic to $\mathbb{R}^N$. This proves that the topological dimension is at least $N$ and since in  general this   is at most the Hausdorff one (see e.g.\ \cite[Theorem 8.14]{Heinonen01}), our claim is proved.

\subsection{Strategy of the proof}

The basic strategy we adopt in proving this theorem is the one introduced by the third named author in \cite{H19} to handle the compact case. Still, moving from compact to non-compact creates additional technical complications that must be handled: one of the things we do here is to replace the approximation of the heat kernel via eigenfunctions - used in \cite{H19} - with suitable decay estimates based on Gaussian bounds. Also, in the course of the proof we obtain (by making explicit some ideas that were implicitly used in \cite{H19}) interesting intermediate results that are new even in the smooth context, see in particular formula \eqref{eq:gtdiv}. Finally, on general $\RCD$ spaces $\X$ of essential dimension $n$ and $U\subset\X$ open we establish relevant links between the properties 
\begin{itemize}
\item[-] $\mathrm{tr}(\mathrm{Hess}f)=\Delta f$ on $U\subset\X$ for every $f$ sufficiently regular,
\item[-] $\mm=c\haus^n$ on $U\subset\X$ for some $c>0$,
\end{itemize} 
see Theorem \ref{thm:main} below for the precise statement.

\bigskip

With this said, let us describe the main idea by having once again a look at the case of a weighted $n$-dimensional Riemannian manifold $(M, g, e^{-V}\di \mathrm{Vol}_g)$. Let us consider the reference measure $\mm:=e^{-V} \mathrm{Vol}_g$ and the Hausdorff measure $\haus^n= \mathrm{Vol}_g$. Assume that ${\rm Ric}_N\geq Kg$  for some $K \in \mathbb{R}$ and some $N \in [n, \infty)$ (namely $(M, \dist_g, \mass)$ is an $\RCD(K, N)$ space, recall \eqref{eq:beric} and \eqref{eq:bekn}). Now notice that the following integration by parts formulas hold: for every $f,\phi\in C^\infty_c(M)$ we have

\begin{subequations}
\begin{align}
\label{eq:intparts1}
-\int_M \langle\d f,\d\phi\rangle\,\d\mm&=\int_M \phi\,\Delta f\,\d\mm,\\
\label{eq:intparts2}
-\int_M \langle\d f,\d\phi\rangle\,\d\haus^n&=\int_M \phi\,\mathrm{tr}(\mathrm{Hess}f)\,\d\haus^n.
\end{align}
\end{subequations}
From these identities it is easy to conclude that
\begin{equation}\label{eq:trHD}
\mm=c\haus^n\qquad\Leftrightarrow\qquad \mathrm{tr}(\mathrm{Hess}f)=\Delta f\quad\forall f\in C^\infty_c(M).
\end{equation}
Thus recalling \eqref{eq:fromloctoglob} we see that the desired result will follow if we can show that 
\[
\text{$\theta_N[M,\sfd_g,\mm]<\infty$ a.e.\ implies that $\mathrm{tr}(\mathrm{Hess}f)=\Delta f$ for any smooth function $f$}.
\]
To see this recall that, as already noticed, having $\theta_N[M,\sfd_g,\mm](x)<\infty$ for some point $x\in M$ implies that $M$ is $N$-dimensional (and thus in particular that $N$ is an integer), then recall \eqref{eq:bekn} and the definition \eqref{eq:beric} of the $N$-Ricci curvature tensor.

This establishes the claim in the smooth setting. In the general case we follow the same general ideas, but we have to deal with severe technical complications. Start observing that the analogue of \eqref{eq:intparts1} holds in general $\RCD$ spaces by the very definition of $\Delta$ (see \cite{Gigli12}) and that the line below \eqref{eq:trHD} is known. 
Thus to conclude along the lines above it is sufficient to prove that \eqref{eq:intparts2} holds on $\RCD(K,N)$ spaces of essential dimension $n$. We do not have exactly such result, but have instead the following result which is anyway sufficient to conclude:
\begin{theorem}\label{le:keyintro}
Let $(\X,\sfd,\mm)$ be an $\RCD(K,N)$ space of essential dimension $n$. Let $U\subset\X$ be bounded open and assume that
\begin{equation}
\label{eq:assU}
\inf_{r\in(0,1),x\in U}\frac{\mm(B_r(x))}{r^n}>0.
\end{equation}
Then for every $\phi\in \Lip(\X, \dist)$, $f\in D(\Delta)$ with $\supp(\phi),\supp(f)\subset U$ formula \eqref{eq:intparts2} holds.
\end{theorem}
See Theorem \ref{intbyparts} for a slightly sharper statement. Notice also that by the Bishop-Gromov inequality, assumption \eqref{eq:assU} holds trivially with $n=N$ for any bounded subset $U$ of a weakly non-collapsed $\RCD$ space. Also, the statement above is interesting regardless of the application we just described, and valid also in possibly `collapsed' $\RCD$ spaces.

Thus everything boils down to the proof of such result, and indeed from both the technical and conceptual perspective this is the most important part of our paper. The basic idea for the proof is to perform a smoothing of the metric tensor via heat flow. Let us describe the procedure, introduced in \cite{Berard:1994vy}, in the smooth setting. Consider a compact smooth Riemannian manifold $(M, g,\di \mathrm{Vol}_g)$ and, for every $t>0$, let $\Phi_t:M\to L^2(M,\mathrm{Vol}_g)$ be defined as
\[
\Phi_t(x):=\left( y \mapsto p(x, y, t)\right),
\] 
where $p$ is the heat kernel. We can use this map to pull-back the flat metric $g_{L^2}$ of $L^2(M,\mathrm{Vol}_g)$ and obtain the metric tensor $g_t:=\Phi_t^*g_{L^2}$ that is explicitly given by
\begin{equation}
\label{eq:defgt}
g_t=\int_{M}\di p( \cdot, y, t) \otimes \di p( \cdot, y, t)\di \mathrm{Vol}_g(y) \in C^{\infty}((T^*)^{\otimes 2}M).
\end{equation}
The interesting fact is that, after appropriate rescaling, the tensors $g_t$ converge to the original one $g$. More precisely, we have
\begin{equation}
\label{eq:unifconv}
\|4(8\pi)^{n/2}t^{(n+2)/2}g_t-g\|_{L^{\infty}} \to 0\qquad\text{as}\ t\to 0^+,
\end{equation}
where $n$ denotes the dimension of $M$.
In fact in \cite{Berard:1994vy} more is proved, as it is provided the first order Taylor expansion of $t^{(n+2)/2}g_t$, but this is not relevant for our application. A way to read  this convergence is via the stability of the heat flow under measured-Gromov-Hausdorff convergence of spaces with Ricci curvature uniformly bounded from below; this observation is more recent than \cite{Berard:1994vy}, as it has been made  by the second author in \cite{Gigli10},  still, this is the argument used in the $\RCD$ setting so let us present this viewpoint. It is clear that for $M=\R^n$ the tensor $g_t$ is just a rescaling of the Euclidean tensor. On the other hand, denoting by $M^\lambda$ the manifold $M$ equipped with the rescaled metric tensor $\lambda g$, and by $p^\lambda$ the associated heat kernel, it is also clear that $p(x,y,t)=p^\lambda(x,y,\lambda^{-1}t)$. Thus the asymptotics of $p(x,y,t)$ as $t\to0^+$ corresponds to that of $p^\lambda(x,y,1)$ as $\lambda\to\infty$ and, as said, these kernels converge to the Euclidean ones where the evolution of the metric tensors $g_t$ is trivial.

Coming back to the $\RCD$ setting, we recall that the heat kernel is well-defined in this context \cite{Ambrosio_2014}, and a differential calculus is available in this framework \cite{Gigli14}. Thus the same definition as in \eqref{eq:defgt} can be given and one can wonder whether the same convergence result as in \eqref{eq:unifconv} holds. Interestingly, in this case one has
\begin{equation}
\label{eq:lpconv}
\|t\mm(B_{\sqrt t}(\cdot))g_t-c_n g\|_{L^{p}_{loc}} \to 0\qquad\text{as}\ t\to 0^+,\ \forall p\in[1,\infty)
\end{equation}
for some constant $c_n$ depending only on the essential dimension of $\X$ (this has been proved in \cite{AHPT} for compact $\RCD(K,N)$ spaces, and is generalized here to the non-compact setting). Notice that the loss from convergence in $L^\infty$ to convergence in $L^p_{loc}$ is unavoidable, but unharmful for our purposes. It is important to remark that the factor in front of $g_t$ is now not constant anymore: this has to do with Gaussian gradient estimates for the heat kernel. Now let $U\subset\X$ be open bounded and assume that  $\haus^n$ is a Radon measure on $U$ (this is always the case if \eqref{eq:assU} holds). In this case by standard results about differentiations of measures we have
\[
\lim_{t\to 0^+}\frac{t\mm(B_{\sqrt t}(\cdot))}{t^{\frac{n+2}2}}=c'_n\frac{\d\mm}{\d\haus^n}\qquad\mm\text{-a.e.\ }on\ U.
\]
Thus if \eqref{eq:assU} holds, from \eqref{eq:lpconv} we deduce that 
\begin{equation}
\label{eq:lpconv2}
\|t^{\frac{n+2}2}g_t -c_n''\tfrac{\d\haus^n}{\d\mm} g\|_{L^p(U)}\to 0\qquad\text{as}\ t\to 0^+.
\end{equation}
We couple this information with the following explicit computation of the adjoint  $\nabla^*$ of the covariant derivative of $g_t$:
\begin{equation}\label{eq:gtdiv}
\nabla^*g_t(x)=-\frac{1}{4}\di_x\Delta_xp(x, x, 2t).
\end{equation}
This formula was obtained in \cite{H19} in the compact setting by expanding the heat kernel via eigenfunctions of the Laplacian. This approach does not work in our current framework and we will rather proceed via a somehow more direct approach based on `local' Bochner integration (see Section \ref{subse:local}).

We are almost done: by explicit computations based on Gaussian estimates one can see that
\[
t^{\frac{n+2}2} \di_x\Delta_xp(x, x, 2t)\weakto 0\qquad\text{as}\ t\to 0^+
\]
in a suitable sense, thus coupling this information with \eqref{eq:gtdiv}, \eqref{eq:lpconv2} and the closure of $\nabla^*$ we conclude that
\[
\nabla^*(\tfrac{\d\haus^n}{\d\mm} g)=0\qquad\text{in }U.
\]
This latter equation is a restatement of \eqref{eq:intparts2} for $f,\phi$ with support in $U$, i.e.\ this argument gives Theorem \ref{le:keyintro}, as desired.

We conclude emphasizing that our proof also yields the following result, which is of independent interest and will play a prominent role in the proof of Theorem \ref{thm:mainres}.

\begin{theorem}\label{thm:main}
Let $(\XX,\dist,\mass)$ be an $\RCD(K,N)$ space of essential dimension $n$, and let $U$ be a connected open subset of $\XX$ with 
    \begin{equation}\label{main:uniformbound}
        \inf_{r \in (0, 1),x \in A}\frac{\meas(B_r(x))}{r^n}>0
    \end{equation} 
for any compact subset $A\subseteq U$. Then the following two conditions are equivalent:
\begin{enumerate}
    \item\label{thm:item1} for every $f\in D(\Delta)$,
\[
        \Delta f=\tra(\hess f)\quad\text{$\mass$-a.e.\ on $U$};
\]
    \item\label{thm:item2} for some $c\in (0,\infty)$,
\[
        \mass\mres U=c\haus^n\mres U.
\]
    \end{enumerate}
\end{theorem}

Notice that this has nothing to do with non-collapsing properties and, in particular, it can very well be that the assumption \eqref{main:uniformbound} holds for $U=\X$. Moreover items \ref{thm:item1} and \ref{thm:item2} may hold only on some $U\subsetneq \X$: just consider the case of a weighted Riemannian manifold as before with $V$ constant on $U$ but non-constant outside $U$.
\subsection{Applications}
The following applications seem to be already known to experts if Theorem \ref{thm:mainres} is established (for instance \cite{KapMon19} and \cite{honda2019sphere}).
However for readers' convenience let us give them precisely. Roughly speaking, they are based on a fact that the space of weakly non-collapsed spaces is open in the space of $\RCD(K, N)$ spaces because of the lower semicontinuity of the essential dimensions with respect to pointed measured Gromov-Hausdorff convergence proved in \cite{kitabeppu2017sufficient} (Theorem \ref{thm:loweress}). 

It is known that pointed Gromov-Hausdorff (pGH) and pointed measured Gromov-Hausdorff (pmGH) convergences are metrizable (see for instance in \cite{GMS15}). Thus `$\epsilon$-pGH close'  and `$\epsilon$-pmGH close' make sense as appeared in the following theorem. 
Note that as the sequential compactness of $\RCD(K, N)$ spaces is known (Theorem \ref{thm:moduli}), any such metric determines the same compact topology.

The first application is stated as follows.
\begin{theorem}\label{thm:close}
For any $K \in \mathbb{R}$, any $N \in \mathbb{N}$, any $\delta\in (0,\infty)$ and any $v \in (0, \infty)$, there exists $\epsilon:=\epsilon (K, N, \delta, v) \in (0, 1)$ such that if a pointed $\RCD(K, N)$ space $(\XX, \dist, \meas, x)$ is so  that $(\XX, \dist, x)$ is $\epsilon$-${\rm pGH}$ close to $(\YY, \dist_{\YY}, y)$ for some non-collapsed $\RCD(K, N)$ space $(\YY, \dist_{\YY}, \haus^N)$ with 
\begin{equation}\label{eq:volumenot}
\haus^N(B_1(y))\ge v, 
\end{equation}
then we have $\meas=c\haus^N$ for some $c \in (0, \infty)$, and moreover $\abs{\haus^N(B_1(x))-\haus^N(B_1(y))}< \delta$. 
\end{theorem}
Next application shows that the non-collapsed condition can be recognized from the point of an infinitesimal view. 
\begin{theorem}\label{thm:tangcha}
Let $(\XX,\dist,\mass)$ be an $\RCD(K,N)$ space. If the essential dimension of some tangent cone $(\YY, \dist_{\YY}, \meas_Y, y)$ at some point $x \in \XX$ is equal to $N$, then $\meas=c\haus^N$ for some $c \in (0, \infty)$.
\end{theorem}
Note that the converse implication also holds in Theorem \ref{thm:tangcha}, namely if $(\XX,\dist,\haus^N)$ is a non-collapsed $\RCD(K, N)$ space, then any tangent cone at any point is also a pointed non-collapsed $\RCD(0, N)$ space (see Theorem \ref{thm:pGH}).

The following final application shows that the non-collapsed condition can be also recognized from the asymptotical point of view. Note that the LHS of \eqref{eq:fromfareq} exists by the Bishop-Gromov inequality, and does not depend on the choice of $x\in\XX$.

\begin{theorem}\label{thm:fromfar}
Let $(\XX,\dist,\mass)$ be an $\RCD(0,N)$ space and assume that 
\begin{equation}\label{eq:fromfareq0}
	\sup_{x\in\XX} \mass(B_1(x))<\infty
\end{equation}
and that for some (hence all) $x\in\XX$
\begin{equation}\label{eq:fromfareq}
	\lim_{r\rightarrow\infty}\frac{\mass(B_r(x))}{r^N}>0.
\end{equation}
Then $\meas=c\haus^N$ for some $c \in (0, \infty)$.
\end{theorem}
Notice that the assumption \eqref{eq:fromfareq0} is essential, as this simple example shows: just consider the $\RCD(0,N)$ space $([0,\infty),\dist_\mathbb{R}, x^{N-1}\haus^1)$, which satisfies \eqref{eq:fromfareq} but is clearly not non-collapsed. Conversely, any non-collapsed $\RCD(K,N)$ space $(\XX,\dist,\haus^N)$ satisfies \eqref{eq:fromfareq0}, as a consequence of the Bishop-Gromov inequality and \eqref{eq:densityest1}.

\smallskip\noindent
\textbf{Acknowledgement.}
The authors thank the anonymous referee for the very careful review and several helpful suggestions that improve the presentation. The third named author acknowledges supports of the Grant-in-Aid for Scientific Research (B) of 20H01799 and 
the Grant-in-Aid for Scientific Research (B) of 21H00977.

\section{Preliminaries}\label{sec:2}
Throughout the paper:
\begin{itemize}
\item by \emph{metric measure space} $(\X,\sfd,\mm)$ we always intend a complete and separable metric space equipped with a non-negative Borel measure finite on bounded sets such that $\supp(\mm)=\X$;
\item $C$ denotes a positive constant, that may vary from step to step. Occasionally we may emphasize the parameters on which the constant depends, so that, say, $C(K, N)$ denotes a positive constant depending only on $K$ and $N$;
\item  $\mathrm{Lip}(\XX, \dist)$ (resp.\ $\mathrm{Lip}_b(\XX, \dist)$, resp.\  $\mathrm{Lip}_{bs}(\XX, \dist)$) denotes the set of all Lipschitz (resp.\ bounded Lipschitz, resp.\ Lipschitz with bounded support) functions on a metric space $(\XX, \dist)$;

\item We denote by $\lip f:\X\to[0,\infty]$ the \emph{local Lipschitz constant} of the function $f:\X\to\R$ defined by
\[
\lip f(x):=\limsup_{y\to x}\frac{|f(y)-f(x)|}{\sfd(y,x)}
\]
if $x$ is not isolated and has to be understood as $0$ if $x$ is isolated;
\item $L^p_{{loc}}$  means that the restriction (for functions, tensors and so on) to any compact subset of the domain is $L^p$.
\end{itemize}
\subsection{Definition and heat flow}\label{2dot1sect}
Fix a metric measure space $(\XX, \dist, \meas)$. The Cheeger energy $\Ch:L^2(\XX,\meas)\to [0,\infty]$ is defined by
\begin{equation}\label{eq:defchee}
  \Ch(f):=\inf_{\|f_i-f\|_{L^2}\to 0}
      \left\{ \liminf_{i\to\infty}\int_\XX(\lip f_i)^2\di\meas:\ f_i\in\Lip_b(\XX,\dist)\cap L^2(\XX,\meas)
     \right\}.
\end{equation}
Then, the Sobolev space $H^{1,2}(\XX,\dist,\meas)$ is defined as the finiteness domain of $\Ch$.
By looking at the optimal sequence in \eqref{eq:defchee} one can identify a canonical object $|\D f|$, called the minimal relaxed slope,
which is local on Borel sets (i.e.\ $|\D f_1|=|\D f_2|$ $\meas$-a.e.\ on $\{f_1=f_2\}$) and provides integral representation to $\Ch$, namely
\[
\Ch (f)=\int_\XX|\D f|^2\di\meas\qquad\forall f\in H^{1,2}(\XX,\dist,\meas).
\]
We are now in a position to introduce the definition of $\RCD(K, N)$ spaces (the equivalence of the following definition with the  one proposed in  \cite{Gigli12} is in \cite{AmbrosioGigliSavare12}, \cite{Erbar-Kuwada-Sturm13}, see also \cite{AmbrosioMondinoSavare13}):

\begin{definition}[$\RCD(K,N)$ space]\label{def:rcdkn}
For any $K \in \mathbb{R}$ and any $N \in [1, \infty]$, a metric measure space $(\XX, \dist, \meas)$ is said to be an $\RCD(K, N)$ space if the following four conditions are satisfied.
\begin{enumerate}
  \item There exist $x \in X$ and $C>1$ such that $\meas (B_r(x))\le Ce^{Cr^2}$ holds for any $r>0$.
  \item $\Ch$ is a quadratic form. In this case for  $f_i \in H^{1, 2}(\XX, \dist, \meas)  (i=1, 2)$ we put
    \begin{equation}\notag
    \langle\nabla f_1,\nabla f_2\rangle:=\lim_{\epsilon\to 0}\frac{|\D (f_1+\epsilon f_2)|^2-|\D f_1|^2}{2\epsilon} \in L^1(\XX, \meas).
    \end{equation}
  \item Any $f \in H^{1, 2}(\XX, \dist, \meas)$ with $|\D f| \le 1$ for $\meas$-a.e.\ has a $1$-Lipschitz representative.
  \item For any $f \in D(\Delta)$ with $\Delta f \in H^{1, 2}(\XX, \dist, \meas)$ we have
    \begin{equation}\label{eq:bochner}
      \frac{1}{2}\int_\XX|\D f|^2\Delta \phi \di \meas \ge \int_\XX\phi \left( \frac{(\Delta f)^2}{N}+\langle \nabla \Delta f, \nabla f\rangle +K|\D f|^2\right)\di \meas
    \end{equation}
    for any $\phi \in D(\Delta) \cap L^{\infty}(\XX, \meas)$ with $\phi \ge 0$, $\Delta \phi \in L^{\infty}(\XX, \meas)$, where 
     \begin{align*}
       D(\Delta)
\defeq \bigg\{ &f \in H^{1,2}(\XX,\dist,\meas) : \exists h \in L^2(\XX,\meas) \, \, \text{s.t.}\\ &\qquad\int_{\XX}\langle \nabla f, \nabla \phi \rangle \di \meas = - \int_\XX h \phi \di \meas, \, \, \forall \phi \in H^{1,2}(\XX,\dist,\meas) \bigg\}
     \end{align*}
 and $\Delta f := h$ for any $f \in D(\Delta)$.
\end{enumerate}
\end{definition}
 
We point out that, unless otherwise specified, when we write $\RCD(K,N)$ we implicitly assume $N<\infty$.
Notice that, by the very definition, if $(\XX,\dist,\mass)$ is an $\RCD(K,N)$ space, then $(\XX,a \dist,b \mass)$ is an $\RCD(a^{-2}K, N)$ space for any $a,b\in(0,\infty)$, for $N\in [1,\infty]$.

In the rest of this subsection, let us fix an $\RCD(K, \infty)$ space $(\XX, \dist, \meas)$ and let us introduce the fundamental properties, except for the \textit{second order differential calculus} developed in \cite{Gigli14} which will be treated in Subsection \ref{subsec:diff}. 

First let us recall the \textit{heat flow} associated with $\Ch$
\[
\heat_t:L^2(\XX, \meas) \to L^2(\XX, \meas).
\]
This family of  maps is characterized by the properties: $\heat_tf \to f$ in $L^2(\XX, \meas)$ as $t\to 0^+$, $\heat_t f\in  D(\Delta)$ for any $f\in L^2$, $t>0$ and for any $t>0 $ it holds
\begin{equation}\label{eq:heatfl}
\frac{\di}{\di t}\heat_tf=\Delta \heat_tf\qquad \text{in $L^2(\XX, \meas)$.}
\end{equation}
It will be useful to keep in mind the following a-priori estimates (\cite[Remark 5.2.11]{GP19}):
\begin{equation}\label{eq:classheat1}
    \norm{|\D \heat_t f|}_{L^2}\le \frac{\norm{f}_{L^2}}{\sqrt t}\qquad\qquad    \norm{\Delta\heat_t f}_{L^2}\le \frac{\norm{f}_{L^2}}{t}\qquad\forall f\in L^2(\X,\mm),\forall t>0
\end{equation}
as well as the fact that
\begin{equation}
\label{eq:l2decr}
t\mapsto \|\heat_tf\|_{L^2}\quad\text{ is non-increasing for every }f\in L^2(\X,\mm).
\end{equation}

Then the $1$-\textit{Bakry-\'Emery} estimate proved in \cite[Corollary 4.3]{Savare13} is stated as for any $f \in H^{1, 2}(\XX, \dist, \meas)$,
\begin{equation}\label{eq:1be}
|\D \heat_tf|(x)\le e^{-Kt}|\D f|(x)\qquad \text{for $\meas$-a.e.\ $x \in \XX$,}
\end{equation}
which in particular implies
\begin{equation}\notag
\heat_tf \to f\qquad \text{in $H^{1, 2}(\XX, \dist, \meas)$.}
\end{equation}
It is also worth pointing out that the heat flow $\heat_t$ also acts on $L^p(\XX, \meas)$ for any $p \in [1, \infty]$ with
\begin{equation}\notag
\|\heat_tf\|_{L^p}\le \|f\|_{L^p}\qquad \forall f \in L^p(\XX, \meas).
\end{equation}

%
%
%
%
Finally let us recall that the following $(1, 1)$-Poincaré inequality is satisfied:
\begin{equation}\label{eq:poincare}
\begin{split}
	&\int_{B_r(x)}\Big| f-\frac{1}{\meas (B_r(x))}\int_{B_r(x)}f\di \meas \Big| \di \meas \le 4e^{|K|r^2}r\int_{B_{3r}(x)}|\D f|\di \meas \qquad \forall f \in H^{1, 2}(\XX, \dist, \meas),\forall r>0
\end{split}
\end{equation}
which is also valid for larger class, $\CD(K, \infty)$ spaces.
See \cite{Rajala12} for the detail.

\subsection{Calculus on $\RCD(K, \infty)$ spaces}\label{subsec:diff}

Let $(\XX, \dist, \meas)$ be a metric measure space. We assume that the readers are familiar with the notion of normed module, introduced in \cite{Gigli14}, inspired by the theory developed in \cite{Weaver01}. Here we just recall few basic definitions, mostly to fix the notation. Unless otherwise stated, the material comes form \cite{Gigli14}, \cite{Gigli17}.

A $L^0$-normed module is a topological vector space $\mathscr M$ that is also a module over the commutative ring with unity $L^0(\X,\mm)$, possessing a \emph{pointwise norm}, i.e.\ a map $|\cdot|:\mathscr M\to L^0(\X,\mm)$ such that
\[
|fv+gw|\leq |f||v|+|g||w|\qquad\mm\text{-a.e.,\ }\forall v,w\in\mathscr M,\forall f,g\in L^0(\X,\mm),
\]
and such that the distance
\begin{equation}
\label{eq:distmod}
\sfd_{\mathscr M}(v,w):=\int_\XX 1\wedge|v-w|\,\d\mm'
\end{equation}
is complete and induces the topology of $\mathscr M$, where here $\mm'$ is a Borel probability  measure such that $\mm\ll\mm'\ll\mm$ (the actual choice of $\mm'$ affects the distance but not the topology nor  completeness).

$\mathscr M$ is said to be a Hilbert module provided
\[
|v+w|^2+|v-w|^2=2(|v|^2+|w|^2)\qquad\mm\text{-a.e.,\ }\forall v,w\in\mathscr M
\]
and in this case by polarization we can define a pointwise scalar product as
\[
\langle v,w\rangle:=\tfrac12(|v+w|^2-|v|^2-|w|^2)\qquad\mm\text{-a.e.,\ }\forall v,w\in\mathscr M
\]
that turns out to be $L^0$-bilinear and continuous. The tensor product of two Hilbert modules $\mathscr M_1,\mathscr M_2$ is defined as the completion of the algebraic tensor product as $L^0$-modules w.r.t.\ the distance induced by the pointwise norm that in turn is induced by the pointwise scalar product characterized by
\[
\langle v_1\otimes w_1,v_2\otimes w_2\rangle_{\sf HS}:=\langle v_1,v_2\rangle_1\, \langle   w_1,  w_2\rangle_2.
\]
The pointwise norm and scalar product on a tensor product will often be denoted with the subscript ${\sf HS}$, standing for \emph{Hilbert-Schmidt}. The dual  $\mathscr M^*$ of $\mathscr M$ is defined as the collection of $L^0$-linear and continuous maps $L:\mathscr M\to L^0(\X,\mm)$, is equipped with the natural multiplication by $L^0$ functions $(f\cdot L(v):=L(fv))$ and the pointwise norm
\[
|L|_*:=\esssup\limits_{v:|v|\leq1 \ \mm\text{-a.e.}}L(v).
\]
It is then easy to check that $\mathscr M^*$ equipped with the topology induced by the distance defined as in \eqref{eq:distmod} is a $L^0$-normed module. If $\mathscr M$ is Hilbert, then so is $\mathscr M^*$  and the map sending $v\in \mathscr M$ to $(w\mapsto \langle v,w\rangle)\in\mathscr M^*$ is an isomorphism of $L^0$-modules, called Riesz isomorphism. 

The kind of differential calculus on metric measure spaces we are going to use in this manuscript is based around the following result, that defines both the \emph{cotangent module} and the \emph{differential} of Sobolev functions:
\begin{theorem}
Let $(\X,\sfd,\mm)$ be a metric measure space. Then there is a unique, up to unique isomorphism, couple $(L^0(T^*(\X,\sfd,\mm)),\d)$ such that $L^0(T^*(\X,\sfd,\mm))$ is a $L^0$-normed module, $\d: H^{1,2}(\X,\sfd,\mm)\to L^0(T^*(\X,\sfd,\mm))$ is linear and such that:
\begin{itemize}
\item[1)] $|\d f|=|D f|$ $\mm$-a.e.\ for every $f\in H^{1,2}(\X,\sfd,\mm)$,
\item[2)] $L^0$-linear combinations of elements of the form $\d f$ for $f\in H^{1,2}(\X,\sfd,\mm)$ are dense in $L^0(T^*(\X,\sfd,\mm))$.
\end{itemize}
\end{theorem}
The dual of $L^0(T^*(\X,\sfd,\mm))$ is denoted $L^0(T(\X,\sfd,\mm))$ and called \emph{tangent module}. Elements of $L^0(T^*(\X,\sfd,\mm))$  are called 1-forms and elements of $L^0(T(\X,\sfd,\mm))$  are called vector fields on $\X$.

In this case we shall denote by $\nabla f\in L^0(T(\X,\sfd,\mm))$ the image of $\d f$ under the Riesz isomorphism.

The tensor product of $L^0(T(\X,\sfd,\mm))$ with itself will be denoted $L^0(T^{\otimes 2}(\X,\sfd,\mm))$, similarly for $L^0(T^*(\X,\sfd,\mm))$.  Notice that, rather trivially, $L^0(T^{\otimes 2}(\X,\sfd,\mm))$ and $L^0((T^*)^{\otimes 2}(\X,\sfd,\mm))$ are one the dual of each other, in a natural way. 

For $p\in[1,\infty]$, the collection of 1-forms $\omega$ with $|\omega|\in L^p(\X,\mm)$ (resp.\ $L^p_{loc}(\X,\mm)$) will be denoted $L^p(T^*(\X,\sfd,\mm))$ (resp.\ $L^p_{loc}(T^*(\X,\sfd,\mm))$). Similarly for vector fields and other tensors. Convergence in the spaces $L^p(T^*(\X,\sfd,\mm))$ (resp.\ $L^p_{loc}(T^*(\X,\sfd,\mm))$) is defined in the obvious way.

\bigskip

All this for general metric measure spaces. In the $\RCD(K,\infty)$ case we now recall the definition of the set of \textit{test} functions (introduced in \cite{Savare13}):
\[
\mathrm{Test}F(\XX, \dist, \meas):=\left\{ f \in \mathrm{Lip}(\XX, \dist) \cap D(\Delta) \cap L^{\infty}(\XX, \meas): \Delta f \in H^{1, 2}(\XX, \dist, \meas)\right\}
\]
which is an algebra. It is known (\cite{Savare13}) that $|\nabla f|^2 \in H^{1, 2}(\XX, \dist, \meas)$ for any $f \in \mathrm{Test}F(\XX, \dist, \meas)$, that $\mathrm{Test}F(\XX, \dist, \meas)$ is dense in $(D(\Delta), \| \cdot\|_D)$, where $\|f\|_D^2:=\|f\|_{H^{1, 2}}^2+\|\Delta f\|_{L^2}^2$,  and that if $f \in L^2 \cap L^{\infty}(\XX, \meas)$, then $\heat_tf \in \mathrm{Test}F(\XX, \dist, \meas)$ for any $t>0$.
The following result is proved in \cite{Gigli14}.
\begin{theorem}\label{thm:bochnerhessian} Let $(\XX, \dist, \meas)$ be an $\RCD(K, \infty)$ space.
For any $f \in \mathrm{Test}F(\XX, \dist, \meas)$ there exists a unique $T \in L^2((T^*)^{\otimes 2}(\XX, \dist, \meas))$ such that for all $f_i \in \mathrm{Test}F(\XX, \dist, \meas)$,
\begin{equation}\label{eq:hess}
T(\nabla f_1, \nabla f_2)=\frac{1}{2}\left(\langle \nabla f_1, \nabla \langle \nabla f_2, \nabla f\rangle \rangle + \langle \nabla f_2, \nabla \langle \nabla f_1, \nabla f \rangle \rangle -\langle f, \nabla \langle \nabla f_1, \nabla f_2\rangle \rangle \right)
\end{equation}
holds for $\meas$-a.e.\ $x \in \XX$. Since $T$ is unique, we denote it by $\mathrm{Hess} f$ and call it the \textit{Hessian} of $f$. Moreover 
for any $f \in \mathrm{Test}F(\XX, \dist, \meas)$ and any $\phi \in D(\Delta) \cap L^{\infty}(\XX, \meas)$ with $\Delta \phi \in L^{\infty}(\XX, \meas)$ and $\phi \ge 0$, we have
\begin{equation}\label{eq:bochnerhessi}
\int_{\XX}\phi |\mathrm{Hess} f|_{\sf HS}^2\di \meas \le \int_{\XX} \frac{1}{2}\Delta \phi \cdot |\nabla f|^2- \phi\langle \nabla \Delta f, \nabla f\rangle -K\phi |\nabla f|^2 \di \meas
\end{equation}
and
\begin{equation}\label{eq:bochner2}
\int_{\XX}|\mathrm{Hess}f|_{\sf HS}^2\dd \meas \le \int_{\XX} (\Delta f)^2-K|\nabla f|^2\dd \meas.
\end{equation}
\end{theorem}
Thanks to (\ref{eq:bochnerhessi}) with the density of  $\mathrm{Test}F(\XX, \dist, \meas)$ in $D(\Delta)$, for any $f \in D(\Delta)$ we can also define $\mathrm{Hess}f \in L^2((T^*)^{\otimes 2}(\XX, \dist, \meas))$ with the equality (\ref{eq:hess}), where $\langle \nabla f, \nabla f_i\rangle \in H^{1, 1}(\XX, \dist, \meas)$.

\begin{definition}[Divergence $\mathrm{div}$]\label{def:adj}
Let $(\XX, \dist, \meas)$ be an $\RCD(K, \infty)$ space. Denote by $D(\mathrm{div})$ (resp.\ $D_{{loc}}(\mathrm{div})$) the set of all $V \in L^2(T(\XX, \dist, \meas))$ (resp.\ $V \in L^2_{{loc}}(T(\XX, \dist, \meas))$) for which there exists $f \in L^2(\XX, \meas)$ (resp.\ $f \in L^2_{{loc}}(\XX, \meas)$) such that 
\[
\int_\XX\langle V, \nabla h \rangle \dd \meas= -\int_\XX f h\, \dd \meas \qquad \forall h \in \mathrm{Lip}_{{bs}}(\XX, \dist).
\]
Since $f$ is unique (because $\mathrm{Lip}_{{bs}}(\XX, \dist)$ is dense in $L^2(\XX,\mass)$), we define $\mathrm{div} V\defeq f$.
\end{definition}
Note that for any $f \in H^{1, 2}(\XX, \dist, \meas)$, $f \in D(\Delta)$ if and only if $\nabla f \in D(\mathrm{div})$. Moreover if $f \in D(\Delta)$, then for any $\phi \in \mathrm{Lip}_b(\XX, \dist)$ we have $\phi \nabla f \in D(\mathrm{div})$ with 
\[
\mathrm{div}(\phi \nabla f)=\langle \nabla \phi, \nabla f\rangle +\phi \Delta f.
\]
Recalling that the covariant derivative of $f\dd h$ is given by $\dd f \otimes \dd h +f \mathrm{Hess} h$, the following definition is justified:
\begin{definition}[Adjoint operator $\nabla^*$] Let $(\XX, \dist, \meas)$ be an $\RCD(K, \infty)$ space.
Denote by $D(\nabla^*)$ (resp.\ $D_{{loc}}(\nabla^*)$) the set of all $T \in L^2((T^*)^{\otimes 2}(\XX, \dist, \meas))$ (resp.\ $T \in L^2_{{loc}}((T^*)^{\otimes 2}(\XX, \dist, \meas))$) for which there exists $\eta \in L^2(T^*(\XX, \dist, \meas))$ (resp.\ $\eta \in L^2_{{loc}}(T^*(\XX, \dist, \meas))$) such that 
\begin{equation}\notag
\int_\XX\langle T, \dd f \otimes \dd h+f\, \mathrm{Hess} h \rangle_{\sf HS}\,\dd \meas =\int_\XX\langle \eta, f \di h \rangle \dd \meas \qquad \forall f \in \mathrm{Lip}_{{bs}}(\XX, \dist),\forall h \in D(\Delta).
\end{equation}
 Since $\eta$ is unique (because objects of the form $f\d h$ generate $L^2(T^*(\X,\sfd,\mm))$), we denote it by $\nabla^*T$.
\end{definition}
It follows from a direct calculation that the following holds. See \cite[Proposition 2.18]{honda2021isometric} for the proof.
\begin{proposition}\label{prop:adjoing formula}
Let $(\XX, \dist, \meas)$ be an $\RCD(K, \infty)$ space and let $f \in \mathrm{Test}F(\XX, \dist, \meas)$. Then we have $\dd f \otimes \dd f \in D(\nabla^*)$ with
\[
\nabla^*(\dd f \otimes \dd f)=-\Delta f \dd f -\frac{1}{2}\dd |\dd f|^2.
\]\end{proposition}

\subsection{Structure of $\RCD(K, N)$ spaces and convergence}
Let $(\XX, \dist, \meas)$ be an $\RCD(K, N)$ space for some $K \in \mathbb{R}$ and some $N \in [1, \infty)$. The main purpose of this subsection is to provide a more detailed metric measure structure theory of $(\XX, \dist, \meas)$ we will need later. For our purpose it is enough to discuss the case when $K<0$.

First let us recall the Bishop-Gromov inequality (which is also valid for larger class, so-called $\CD(K,N)$ spaces, see \cite[Theorem 5.31]{Lott-Villani09}, \cite[Theorem 2.3]{Sturm06II}).
\begin{equation}\label{eq:bishgrom}
\frac{\mass(B_R(x))}{\mass(B_r(x))}\le \frac{\int_0^R \sinh(t\sqrt{\frac{-K}{N-1}})^{N-1}\dd{t}} {\int_0^r \sinh(t\sqrt{\frac{-K}{N-1}})^{N-1}\dd{t}}\qquad\forall x \in \XX,\forall r< R,
\end{equation}
where, in the case $N=1$, $\sinh(t\sqrt{\frac{-K}{N-1}})^{N-1}$ has to be interpreted as $1$.
It then follows from (\ref{eq:bishgrom}) that 
\begin{equation}\label{eq:growthes}
\frac{\mass(B_R(x))}{{\mass(B_r(x))}}\le C(K,N) \exp\left(C(K,N)\frac{R}{r}\right)\qquad\forall x \in \XX,\forall r< R
\end{equation}
and
\begin{equation}\label{eq:boundmeasballs}
\frac{\mass(B_r(x))}{{\mass(B_r(y))}}\le C(K,N) \exp\left(C(K,N)\frac{\dist(x,y)}{r}\right)\qquad \forall x, y\in\XX, \forall r>0
\end{equation}
are satisfied. It is well-known that from the Bishop-Gromov inequality it follows that the metric structure $(\XX, \dist)$ is proper, hence geodesic, being $(\XX, \dist)$ a length space. The length space property of $\RCD$ spaces follows quite easily from the so called \textit{Sobolev to Lipschitz} property, namely item 3) of Definition \ref{def:rcdkn} (e.g.\ \cite[Theorem 3.10]{AmbrosioGigliSavare12} and references therein).

The following elementary lemma will play a role later.
\begin{lemma}\label{lem:volume}
Let $(\XX,\dist,\meas)$ be an $\RCD(K, N)$ space. Then for any $t \in (0, 1]$, any $\alpha \in \mathbb{R}$, any $\beta \in (0, \infty)$ and any $x \in \XX$ we have
\begin{equation}\label{lem:bound}
\int_{\XX} \meas (B_{\sqrt{t}}(y))^{\alpha} \exp \left(-\frac{\beta \dist^2 (x, y)}{t}\right)\di\meas (y) \le C(K, N, \alpha, \beta )\meas (B_{\sqrt{t}}(x)))^{\alpha+1}.
\end{equation}
\end{lemma}

\begin{proof}
Considering a rescaling $\sqrt{\beta /t} \cdot \dist$ with (\ref{eq:bishgrom}), it is enough to prove (\ref{lem:bound})  assuming  $\beta=t=1$. Then by (\ref{eq:growthes}) and (\ref{eq:boundmeasballs})
\begin{align}
&\int_{\XX} \meas (B_{1}(y))^{\alpha} \exp \left(-\dist^2 (x, y)\right)\di\meas (y) \nonumber \\
&=\sum_{j=-\infty}^{\infty}\int_{B_{2^{j+1}}(x) \setminus B_{2^j}(x)}\meas (B_{1}(y))^{\alpha} \exp \left(-\dist^2 (x, y)\right)\di\meas (y) \nonumber \\
&\le C(K, N)\meas (B_1(x))^{\alpha} \sum_{j=-\infty}^{\infty}\int_{B_{2^{j+1}}(x) \setminus B_{2^j}(x)}\exp \left( C(\alpha, K, N)2^{j+1}-2^{2j}\right)\di \meas(y) \nonumber \\
&= C(K, N)\meas (B_1(x))^{\alpha} \sum_{j=-\infty}^{\infty}\meas (B_{2^{j+1}}(x) \setminus B_{2^j}(x))\exp \left( C(\alpha, K, N)2^{j+1}-2^{2j}\right) \nonumber \\
&\le C(K, N) \meas (B_1(x))^{\alpha} \sum_{j=-\infty}^{\infty} \meas (B_1(x)) \cdot  \exp \left( C(K, N)2^j\right) \cdot \exp \left( C(\alpha, K, N)2^{j+1}-2^{2j}\right) \nonumber \\
&\le C(\alpha, K, N) \meas (B_1(x))^{\alpha+1}. \notag\qedhere
\end{align}
\end{proof}

For the definition of  \textit{pointed measured Gromov-Hausdorff convergence} and the following compactness result we refer, for instance, to \cite[Section 3]{GMS15}.

\begin{theorem}\label{thm:moduli}
If a sequence of pointed $\RCD(K, N)$ spaces $(\XX_i, \dist_i, \meas_i, x_i)$ satisfies
\[
0<\liminf_{i \to \infty}\meas_i(B_1(x_i)) \le \limsup_{i \to \infty}\meas_i(B_1(x_i))<\infty,
\]
then the sequence has a subsequence $(\XX_{i_j}, \dist_{i_j}, \meas_{i_j}, x_{i_j})$ ${\rm pmGH}$ converging  to a pointed $\RCD(K, N)$ space $(\X,\sfd,\mm,x)$. 
\end{theorem}

Next we introduce the notion of \textit{tangent cones} 

\begin{definition}[Tangent cones]\label{def:tangen}Let $(\XX,\dist,\meas)$ be an $\RCD(K, N)$ space.
For $x \in\XX$, we denote by $\mathrm{Tan}(\XX, \dist,\meas,x)$ the set of tangent cones to $(\XX,\dist,\meas)$ at $x$: the collection of all isomorphism  classes of 
pointed metric measure spaces $(\YY, \dist_{\YY},\meas_{\YY},y)$ such that, as $i\to\infty$, one has
\begin{equation}\label{eq:ulla_ulla}
\left(\XX, \frac{1}{r_i}\dist, \frac{1}{\meas (B_{r_i}(x))}\meas, x\right) \stackrel{\mathrm{pmGH}}{\to} (\YY, \dist_{\YY},\meas_{\YY},y) 
\end{equation}
for some $r_i\to 0^+$.
\end{definition}

Note that Theorem \ref{thm:moduli} proves $\mathrm{Tan}(\XX, \dist,\meas,x) \neq \emptyset$ for any $x \in \XX$.
We are now in a position to introduce the key notions of   \textit{regular sets} and the \textit{essential dimension} as follows.

\begin{definition}[Regular set $\mathcal{R}_k$]Let $(\XX,\dist,\meas)$ be an $\RCD(K, N)$ space.
For any $k \geq 1$, we denote by $\mathcal{R}_k$ the $k$-dimensional regular set  of $(X, \dist, \meas)$, 
namely the set of points $x \in\XX$ such that
\[
\mathrm{Tan}(X, \dist,\meas,x) =\left\{ \left(\mathbb{R}^k, \dist_{\mathbb{R}^k},(\omega_k)^{-1}\haus^k,0_k\right) \right\},
\]
where $\omega_k$ is the $k$-dimensional volume of the unit ball in $\mathbb{R}^k$ with respect to the $k$-dimensional Hausdorff measure $\haus^k$. 
\end{definition}

The following result is proved in \cite[Theorem 0.1]{bru2018constancy}.

\begin{theorem}[Essential dimension]\label{th: RCD decomposition} Let $(\X,\sfd,\mm)$ be an $\RCD(K,N)$ space. Then there exists a unique integer $n\in [1,N]$, called the \textit{essential dimension of $(\XX, \dist, \meas)$}, denoted by $\mathrm{essdim}(\XX)$, such that
\[
\meas(X\setminus \mathcal{R}_n\bigr)=0.
\]
\end{theorem}

\begin{remark}\label{rem:essmetric}
The essential dimension is a purely metric concept, actually it is equal to the maximal number $n \in \mathbb{N}$ satisfying
\begin{equation}\notag
\left(\XX, \frac{1}{r_i}\dist, x\right)  \stackrel{\mathrm{pGH}}{\to}(\mathbb{R}^n, \dist_{\mathbb{R}^n}, 0_n)
\end{equation}
for some $x \in \XX$ and some $r_i \to 0^+$ because of the splitting theorem \cite[Theorem 1.4]{Gigli13} and the phenomenon of \emph{propagation of regularity}. See  \cite[Remark 4.3]{kitabeppu2017sufficient}, \cite[Proposition 2.4]{honda2019sphere} and \cite{BruPasSem20}. \fr
\end{remark}

Next let us introduce  a relationship between $\meas$ and the Hausdorff measure of the essential dimension. See \cite{AHT17, DPMR16, GP16-2, KelMon16} for the detail.

\begin{theorem}\label{thm:RN}
Let $(\X,\sfd,\mm)$ be an $\RCD(K,N)$ space and let $n$ be its   essential dimension. Then $\mm\ll \haus^n\res\mathcal{R}_n$. Also, letting  $\meas=\theta\haus^n\res\mathcal{R}_n$ and 
\begin{equation}\label{eq:rnstar}
{\mathcal R}_n^*:=\left\{x\in\mathcal{R}_n:\
\exists\lim_{r\rightarrow 0^+}\frac{\meas(B_r(x))}{\omega_n r^n}\in (0,\infty)\right\}
\end{equation}
we have that  $\mass(\mathcal{R}_n\setminus\mathcal{R}_n^*)=0$, $\mass\res\mathcal{R}_n^*$ and 
$\haus^n\res\mathcal{R}_n^*$ are mutually absolutely continuous and
\[
\lim_{r\rightarrow 0^+}\frac{\meas(B_r(x))}{\omega_n r^n}=\theta(x)
\qquad\text{for $\mass$-a.e.\ $x\in\mathcal{R}_n^*$.}
\]
Moreover  $\haus^n(\mathcal{R}_n\setminus\mathcal{R}_n^*)=0$ 
if $n=N$.
\end{theorem}
A more general and classical result  concerning densities, that we shall use later on, is  the following  (see e.g.\ \cite[Theorem 2.4.3]{AmbrosioTilli04} for a proof):
\begin{lemma}\label{le:densh}
Let $(\X,\sfd,\mm)$ be a metric measure space, $\alpha\geq 0$ and $A\subset \X$ a Borel subset such that
\[
\limsup_{r\to 0^+}\frac{\mm(B_r(x))}{r^\alpha}>0\qquad\forall x\in A.
\]
Then $\haus^\alpha\mres A$ is a Radon measure absolutely continuous w.r.t.\ $\mm$.
\end{lemma}
The fact that $L^0(T(\X,\sfd,\mm))$ is a Hilbert module is an indication of the existence of some (weak) Riemannian metric on $\X$. This statement can easily be made more explicit by building upon the fact that such module has local dimension equal to the essential dimension of $\X$ (see  \cite{GP16}):
\begin{proposition}\label{prop:riemannianmetric}
Let $(\X,\sfd,\mm)$ be an $\RCD(K,N)$ space of essential dimension $n$. Then there is a unique $g\in L^0((T^*)^{\otimes 2}(\XX, \dist, \meas))$ such that 
\[
g(V_1\otimes V_2)=\langle V_1,V_2\rangle\qquad\mm\text{-a.e.,\ }\forall V_1, V_2\in L^0(T(\X,\sfd,\mm)).
\]
Moreover, $g$ satisfies 
\begin{equation}
\label{eq:normg}
|g|_{\mathrm{HS}}=\sqrt{n},\qquad \meas\text{-a.e.}.
\end{equation}
\end{proposition}
We can use this `metric tensor' to  define the \textit{trace} of  any $T \in L^0((T^*)^{\otimes 2}(\XX, \dist,  \meas))$ by
\[
\mathrm{tr}(T):=\langle T, g\rangle_{\sf HS} \in L^0(\XX, \meas).
\]
Notice that by \eqref{eq:normg} and the Cauchy-Schwarz inequality it follows that if $T\in L^p_{loc}((T^*)^{\otimes 2}(\XX, \dist,  \meas))$, then $\mathrm{tr}(T)\in L^p_{loc}(\X,\sfd,\mm)$.

Finally let us end this subsection by recalling the lower semicontinuity of the essential dimensions with respect to pmGH convergence proved in \cite[Theorem 1.5]{kitabeppu2017sufficient}, where this is also understood as a consequence of $L^2_{{loc}}$-weak convergence of Riemannian metrics (see \cite[Remark 5.20]{AHPT}), and an alternative proof of the theorem below can be based on Remark \ref{rem:essmetric}.
\begin{theorem}\label{thm:loweress}
Let 
\[
(\XX_i, \dist_i, \meas_i, x_i) \stackrel{\mathrm{pmGH}}{\to} (\XX, \dist, \meas, x) 
\]
be a ${\rm pmGH}$ convergent sequence of pointed $\RCD(K, N)$ spaces. Then
\[
\liminf_{i \to \infty}\mathrm{essdim}(\XX_i) \ge \mathrm{essdim}(\XX).
\]
\end{theorem}

\subsection{Non-collapsed $\RCD(K, N)$ spaces}
Let us start recalling the following:
\begin{definition}[Non-collapsed $\RCD(K, N)$ space]\label{def:noncolla}
An $\RCD(K, N)$ space $(\X,\sfd,\mm)$ is said to be \textit{non-collapsed} if $\meas=\haus^N$.
\end{definition}
This definition was introduced in \cite[Definition 1.1]{DPG17} as a synthetic counterpart of non-collapsed Ricci limit spaces. As explained in the introduction, non-collapsed $\RCD(K, N)$ spaces have finer properties rather than general $\RCD(K, N)$ spaces already introduced in subsection \ref{subsec:diff}. Let us give one of the properties as follows (see \cite[Theorem 1.2]{DPG17}).

\begin{theorem}[From pGH to pmGH]\label{thm:pGH}
Let $K\in\R$, $N\in\N$ and $(\XX_i, \dist_i, \haus^N, x_i)$ be a sequence of pointed non-collapsed $\RCD(K, N)$ spaces. Then after passing to a subsequence, there exists a pointed proper geodesic space $(\XX, \dist, x)$ such that 
\[
(\XX_i, \dist_i, x_i) \stackrel{\mathrm{pGH}}{\to} (\XX, \dist, x).
\]
Moreover, if $\inf_i\haus^N(B_1(x_i))>0$, then $(\XX, \dist,\haus^N, x)$ is also a pointed non-collapsed $\RCD(K, N)$ space and the convergence of the $(\XX_i, \dist_i, \haus^N, x_i)$'s to such space is in the ${\rm pmGH}$ topology.
\end{theorem}
We remark that the above theorem is tightly related to the following continuity result, which is the generalization to the $\RCD$ class of the classical statement by Colding about volume convergence under lower Ricci bounds \cite{Colding97} (see \cite[Theorem 1.3]{DPG17}):
\begin{theorem}[Continuity of $\haus^N$]\label{thm:conthN}
For $K\in\R$, $N\in[1,\infty)$ let $\mathbb B(K,N)$ be the collection of (isometry classes of) open unit balls on $\RCD(K,N)$ spaces. Equip $\mathbb B(K,N)$ with the Gromov-Hausdorff distance.

Then the map $\mathbb B(K,N)\ni B\mapsto \haus^N(B)\in\R$ is continuous.
\end{theorem}

For our main purpose, we need a  notion weaker than the non-collapsed one. In order to give the precise definition, let us recall the following result which is just a combination from previous known ones:
\begin{theorem}\label{thm:equivnon}
Let $(\XX, \dist, \meas)$ be an $\RCD(K, N)$ space. Then the following five conditions are equivalent.
\begin{enumerate}
\item\label{item1weak} The essential dimension of $\X$ is $N$.
\item\label{item2weak} $\meas$ is absolutely continuous with respect to $\haus^N$.
\item\label{item3weak} (\ref{eq:positive}) holds.
\item\label{item4weak} $N\in\N$ and the Hausdorff dimension of $(\XX, \dist)$ is greater than $N-1$.
\item \label{item5weak}The Hausdorff dimension of $(\XX, \dist)$ is $N$.
\end{enumerate}
\end{theorem}
\begin{proof}
The equivalence between item \ref{item1weak} and item \ref{item2weak} is proved in \cite[Theorem 1.12]{DPG17}. Since the implication from item \ref{item2weak} to item \ref{item3weak} is a direct consequence of Theorem \ref{thm:RN}, let us check the implication from item \ref{item3weak} to \ref{item1weak} as follows.
The positivity (\ref{eq:positive}) with Theorems \ref{th: RCD decomposition} and  \ref{thm:RN} yields 
\[
\mathcal{H}^N(\mathcal{R}_n^*)>0,
\]
where $n$ denote the essential dimension. In particular $N \le n$. Since the converse inequality is always satisfied by Theorem \ref{thm:RN}, we have item \ref{item1weak}.

Notice that item \ref{item2weak} implies item \ref{item4weak}, we show now that item \ref{item4weak} implies  item \ref{item1weak}. To see this, notice that the proof of \cite[Theorem 1.4]{DPG17} shows that if  item \ref{item4weak} holds, then there is an iterated tangent space isomorphic to $\R^N$. Since the essential dimension of the $N$-dimensional Euclidean space is $N$, the conclusion follows from Theorem \ref{thm:loweress}. 

If we assume item \ref{item5weak}, then, since the Hausdorff dimension of $(\XX,\dist)$ is at most the integer part of $N$  (by \cite[Corollary 1.5]{DPG17}), we see that $N$ is an integer so that item \ref{item4weak} holds. Finally, if item \ref{item2weak} holds, then the Hausdorff dimension of $(\XX,\dist)$ is at least $N$, so that we conclude by \cite[Corollary 1.5]{DPG17} again.
\end{proof}
We are now in a position to introduce the notion of weakly non-collapsed $\RCD(K, N)$ spaces (our definition is trivially equivalent to the one in \cite{DPG17}):
\begin{definition}[Weakly non-collapsed $\RCD(K, N)$ space]\label{def:weak}
An $\RCD(K, N)$ space is said to be \textit{weakly non-collapsed} if one (and thus any) of the items in Theorem \ref{thm:equivnon} is satisfied.
\end{definition}
Note that any non-collapsed $\RCD(K,N)$ space is a weakly non-collapsed $\RCD(K,N)$ space.

We conclude the section recalling - see e.g.\ the introduction - that one expects the notion of non-collapsed space to be related to the fact that the trace of the Hessian is the Laplacian. A first instance of this behaviour is contained in the following result, that is basically extracted from \cite[Proposition 3.2]{Han14} (notice that Definition \ref{def:rcdkn} tells that if the stated inequality \eqref{eq:locbo} holds without restrictions on the support of $\phi$, then the space is an $\RCD(K,n)$ space and thus, since $n$ is assumed to be the essential dimension, the space is weakly non-collapsed).
\begin{theorem}\label{thm:generalizedhan} Let $(\X,\sfd,\mm)$ be an $\RCD(K,N)$ space of essential dimension $n$ and let $U\subset\X$ be open. Then the following two conditions are equivalent:

\begin{enumerate}
\item\label{item1han}  For any $f \in \mathrm{Test}F(\XX, \dist, \meas)$ and any $\phi \in D (\Delta)$ non-negative with $\supp(\phi)\subset  U$ and  $\Delta \phi \in L^{\infty}(\XX, \meas)$  we have
\begin{equation}\label{eq:locbo}
\frac{1}{2}\int_U\Delta \phi \, |\nabla f|^2\di \meas \ge \int_U\phi \left(\frac{(\Delta f)^2}{n}+\langle \nabla \Delta f, \nabla f\rangle +K|\nabla f|^2\right) \di \meas.
\end{equation}
\item\label{item2han} For any $f \in D(\Delta)$ we have
\begin{equation}
\label{eq:deltatrH}
\Delta f=\mathrm{tr}(\mathrm{Hess} f)\qquad \text{ $\meas$-a.e.\ in $U$.}
\end{equation}
\end{enumerate}
\end{theorem}
\begin{proof}
It is easy to see the implication from item \ref{item2han} to item \ref{item1han} is trivial because we know
\[
|\mathrm{tr}(\mathrm{Hess} f)|=|\langle \mathrm{Hess}f, g\rangle_{\sf HS}| \le |\mathrm{Hess}f|_{\sf HS}\,|g|_{\sf HS} = |\mathrm{Hess}f|_{\sf HS} \cdot \sqrt{n}.
\]
Thus item \ref{item2han}  gives  $|\mathrm{Hess}f|_{\sf HS}^2 \ge (\Delta f)^2/n$, and therefore item \ref{item1han} follows directly from (\ref{eq:bochnerhessi}).

For the reverse implication we  closely follow  the proof of \cite[Proposition 3.2]{Han14} keeping  in mind  \eqref{eq:locbo} and the existence, for any $A\subset U$ with $A$ compact and $U$ open, of a test function identically 1 on $A$ and with support in $U$ (see e.g.\ \cite{AmbrosioMondinoSavare13-2} or \cite[Lemma 6.2.15]{GP19}). In this way we easily obtain  that \eqref{eq:deltatrH} holds for any $f\in \mathrm{TestF}(\XX,\dist,\meas)$.  Then by the density of $\mathrm{Test}F(\XX,\dist,\meas)$ in $D(\Delta)$ (see for example \cite[Lemma 2.2]{H19})  \eqref{eq:deltatrH} holds for $f\in D(\Delta)$. 
\end{proof}

\section{Smoothing of the Riemannian metric by the heat kernel}\label{sec:3}

\subsection{Local Hille's theorem}\label{subse:local}

In this section we collect some basic results about local (differentiation) operators: the main result we have in mind is the version of Hille's theorem stated in Lemma \ref{le:hille} below. We shall apply the notions presented here to the operators $\d,\Delta,\nabla^*$, but in order to highlight the similarities among the various approaches we shall give a rather abstract presentation.

Thus let us fix a metric measure space $(\X,\sfd,\mm)$ and two $L^0$-normed modules $\mathscr M,\mathscr N$. For $p\in[1,\infty]$ we shall denote by $L^p(\mathscr M)$ (resp.\ $L^p_{loc}(\mathscr M)$) the collection of those $v\in\mathscr M$ with $|v|\in L^p(\X,\mm)$ (resp.\ $|v|\in L^p_{loc}(\X,\mm)$). Similarly for $\mathscr N$.

\begin{definition}[Weakly local operators]
Let $p\in[1,\infty]$ and $L:D(L)\subset L^p(\mathscr M)\to L^p(\mathscr N)$ be a linear operator. We say that $L$ is \textit{weakly local} provided 
\[
L(v)=L(w)\qquad\mm\text{-a.e.\ on the essential interior of $\{v=w\}$ for any }v,w\in D(L).
\]
\end{definition}
In other words, $L$ is weakly local provided for any $v,w\in D(L)$ and $U\subset\X$ open such that $v=w$ $\mm$-a.e.\ on $U$, we have $L(v)=L(w)$ $\mm$-a.e.\ on $U$.

Weakly local operators can naturally be extended as follows (variants of this definition are possible, but for us the following is sufficient):
\begin{definition}[Extension of weakly local operators]
Let $p\in[1,\infty]$ and $L:D(L)\subset L^p(\mathscr M)\to L^p(\mathscr N)$ be a weakly local operator. We then define $D_{loc}(L)\subset L^p_{loc}(\mathscr M)$ as the collection of those $v$'s such that for every $U\subset\X$ bounded and open there is $v_U\in D(L)\subset L^p(\mathscr M)$   with $v_U=v$ $\mm$-a.e.\ on $U$.

For $v\in D_{loc}(L)$ we define $L(v)\in L^p_{loc}(\mathscr N)$ via
\[
L(v)=L(v_U)\qquad\mm\text{-a.e.\ on}\ U,\forall U\subset\X\text{ open and bounded},
\]
where $v_U$ is as above.
\end{definition}
It is clear from  the definition that $L:D_{loc}(L)\subset L^p_{loc}(\mathscr M)\to L^p_{loc}(\mathscr N)$ is well-posed and that the resulting operator is linear. We are interested in a version of Hille's theorem for this kind of operators and to this aim we need first to introduce the notion of integrable function with values in $L^p(\mathscr M)$.

For the standard notion of Bochner integration of Banach valued maps we refer to \cite{DiestelUhl77}. Given a metric measure space $(\Y,\sfd_\Y,\mu)$ (the topology here is not really relevant, but in our applications we shall mostly have $\Y=\X$) we shall denote by $L^1(\Y,\mu;L^p_{loc}(\mathscr M))$ the collection of (equivalence classes up to $\mu$-a.e.\ equality of) maps $y\mapsto v_y\in L^p_{loc}(\mathscr M)$ such that for any $A\subset \X$ Borel and bounded the map $y\mapsto \chi_A v_y$ is in $L^1(\Y,\mu;L^p(\mathscr M))$ (here we are endowing $L^p(\mathscr M)$ with its natural Banach structure).

With these definitions, the following result is rather trivial (but nevertheless useful):

\begin{lemma}[Local Hille's theorem - abstract version]\label{le:hille}
Let $(\X,\sfd,\mm)$, $(\Y,\sfd_\Y,\mu)$ be  metric measure spaces,  $\mathscr M,\mathscr N$ two $L^0$-normed modules, $p\in[1,\infty]$ and $L:D(L)\subset L^p(\mathscr M)\to L^p(\mathscr N)$ a weakly local and closed linear operator. Also, let $y\mapsto v_y\in L^p_{loc}(\mathscr M)$ be in $L^1(\Y,\mu;L^p_{loc}(\mathscr M))$. Assume that
\begin{itemize}
\item[i)] $v_y\in D_{loc}(L)$ for $\mu$-a.e.\ $y$,
\item[ii)] there exists a `cut-off' operator $T$:  for every $V\subset U\subset\X$ bounded and open with $\sfd(V,\X\setminus U)>0$ there is a linear  map $T:L^p_{loc}(\mathscr M)\to L^p(\mathscr M)$ such that:
\begin{equation}
\label{eq:cutoffT}
\begin{split}
T(v)&=v\qquad \mm\text{-a.e.\ on\ }V,\\
 T(v)&=T(\chi_Uv)\qquad\mm\text{-a.e.}\\
 \|T(v)\|_{L^p(\mathscr M)}&\leq C\|\chi_U v\|_{L^p(\mathscr M)}
 \end{split}
\end{equation}
for every $v\in L^p_{loc}(\mathscr M)$ and some $C>0$ independent on $v$,
\item[iii)] $L$ has the following `stability under cut-off by $T$' property:  for any $V,U$ as above and $T$ given by item $(ii)$  we have $T(v_y)\in D(L)$ for $\mu$-a.e.\ $y$ and the map  $y\mapsto L(T(v_y))$ is in $L^1(\Y,\mu;L^p(\mathscr N))$.
\end{itemize}
Then $\int_{\YY} v_y\,\d\mm(y)\in D_{loc}(L)$, the map $y\mapsto L(v_y)$ is in $L^1(\Y,\mu;L^p_{loc}(\mathscr M))$ and
\[
L\Big(\int_\YY v_y\,\d\mm(y)\Big)=\int_\YY L(v_y)\,\d\mm(y).
\]
\end{lemma}

\begin{proof} Fix $V\subset \X$ open bounded and then let $U\supseteq V$ open bounded be with $\sfd(V,\X\setminus U)>0$.  Let $T:L^p_{loc}(\mathscr M)\to L^p(\mathscr M)$ be given by item $(ii)$. By the assumption $(i)$ we know that $y\mapsto \chi_U v_y\in L^p(\mathscr M)$ is in $L^1(\X,\mm;L^p(\mathscr M))$ and the  third in \eqref{eq:cutoffT} gives that $T$ is continuous as map from $L^p(\mathscr M)$ to itself. It follows that $y\mapsto T(\chi_Uv_y)=T(v_y)$ is in $L^1(\Y,\mu;L^p(\mathscr M))$. Then assumption $(iii)$ and the classical theorem by Hille ensure that 
\begin{equation}
\label{eq:hilleclas}
\int_{\YY} T(v_y)\,\d\mu(y)\in D(L)\qquad\text{ and }\qquad L\Big(\int_\YY T(v_y)\,\d\mu(y)\Big)=\int_\YY L(T(v_y))\,\d\mu(y).
\end{equation}
Now notice that the first in \eqref{eq:cutoffT} give that $T(v_y)=v_y$ on $V$ for every $y$, thus the weak locality of $L$ also gives that $L(T(v_y))=L(v_y)$ on $V$ for every $y$. It follows that $y\mapsto \chi_VL(v_y)$ is in $L^1(\Y,\mu;L^p(\mathscr N))$ and that
\[
\int_\YY \chi_V v_y\,\d\mu(y)=\int_\YY\chi_V  T(v_y)\,\d \mu(y)\qquad\text{and}\qquad\int_\YY \chi_V L(v_y)\,\d\mu(y)=\int_\YY \chi_V  L(T(v_y))\,\d \mu(y).
\]
Thus using again the weak locality of $L$ it follows that
\[
\begin{split}
\chi_VL\Big(\int_\YY v_y\,\d \mu(y)\Big)&=\chi_VL\Big(\int_\YY T(v_y)\,\d \mu(y)\Big)\\
\textrm{(by \eqref{eq:hilleclas})}\qquad&=\chi_V\int_\YY L(T(v_y))\,\d\mu(y)=\int_\YY \chi_V L(T(v_y))\,\d\mu(y)=\int_\YY \chi_V L(v_y)\,\d\mu(y).
\end{split}
\]
Since $V$ was arbitrary, this is the conclusion.
\end{proof}

We now see how to apply this general statement to the concrete cases of $L=\d,\Delta,\nabla^*$. The idea is to use, as map $T$, the multiplication with a Lipschitz cut-off function $\phi$ with support in $U$ and identically 1 on $V$. For the case of the Laplacian this does not really work, as one would need to multiply by a Lipschitz function with bounded Laplacian in order to remain in the domain of the operator. The problem is that on general $\RCD(K,\infty)$ spaces it is not clear whether this sort of cut-off functions exist (but see \cite{AmbrosioMondinoSavare13-2} or \cite[Lemma 6.2.15]{GP19} for the case of proper $\RCD$ spaces). This issue is, however, easily dealt with by recalling that the Laplacian is the divergence of the gradient and applying the above theorem twice (this amount at localizing $\Delta f$ by looking at $\div(\phi \nabla(\phi f))$).

\medskip

Let us start recalling that the differential $\d :H^{1,2}(\X,\sfd,\mm)\subset L^2(\X,\mm)\to L^2(T^*(\X,\sfd,\mm))$ is weakly local (in fact even more, as there is locality on Borel sets and not just on open ones) by \cite[Theorem 2.2.3]{Gigli14}. The same holds for the divergence operator $\div:D(\div)\subset L^2(T(\X,\sfd,\mm))\to L^2(\X,\mm)$. Indeed, for $v,w\in D(\div)$ equal on some open set $U$, we have
\[
\int_\XX \phi \div v\,\d\mm=-\int_\XX \d\phi(v)\,\d\mm\stackrel{(*)}=-\int_\XX \d\phi(w)\,\d\mm=\int_\XX \phi \div w\,\d\mm
\]
for any $\phi\in \Lip(\X,\dist)$   with $\supp(\phi)\subset U$, having used the locality of the differential and the assumption $v=w$ on $U$ in the starred equality $(*)$. This is sufficient to prove the claim. Similarly, starting from the locality of the covariant derivative (see \cite[Proposition 3.4.9]{Gigli14}) it follows the weak locality of $\nabla^*$. Finally, the weak locality of the Laplacian follows from that of the differential and of the divergence.

Below for the domain $D_{loc}(\d)\subset L^2_{loc}(\X,\mm)$ we shall use the more standard notation $H^{1,2}_{loc}(\X,\sfd,\mm)$.
We then have the following.

\begin{proposition}[Local Hille's theorem - concrete version]\label{prop:lochille} Let $(\X,\sfd,\mm)$ be an $\RCD(K,\infty)$ space and $(\Y,\sfd_\Y,\mu)$ a metric measure space. Let $(f_y)\in L^1(\Y,\mu;L^2_{loc}(\X,\mm))$ (resp.\ $(f_y)\in L^1(\Y,\mu;L^2_{loc}(\X,\mm))$, resp.\ $(A_y)\in L^1(\Y,\mu;L^2_{loc}(T^{\otimes 2}(\X,\sfd,\mm)))$) be with $f_y\in H^{1,2}_{loc}(\X,\sfd,\mm)$ (resp.\ $f_y\in D_{loc}(\Delta)$, resp.\ $A_y\in D_{loc}(\nabla^*)$) for $\mu$-a.e.\ $y\in\Y$.

Then for every $U\subset\X$ open bounded we have that $y\mapsto \chi_U\d f_y$ (resp.\ $y\mapsto  \chi_U\Delta f_y$, resp.\ $y\mapsto  \chi_U\nabla^*A_y$) is - the equivalence class up to $\mu$-a.e.\ equality of - a strongly Borel function (i.e.\ Borel and essentially separably valued).

Now assume also that for every $U\subset \X$ open bounded we have $\int_\YY  \|\chi_U|\d f_y|\|_{L^2}\,\d\mu(y)<\infty$ (resp.\ $\int_\YY  \|\chi_U\Delta f_y\|_{L^2}\,\d\mu(y)<\infty$, resp.\ $\int_\YY \|\chi_U|\nabla^*A_y|\|_{L^2}\,\d\mu(y)<\infty$).

Then  $\int_\YY f_y\,\d \mu(y)\in H^{1,2}_{loc}(\X,\sfd,\mm)$ (resp.\ $\int_\YY f_y\,\d \mu(y)\in D_{loc}(\Delta)$, resp.\ $\int_\YY A_y\,\d \mu(y)\in D_{loc}(\nabla^*)$) with
\[
\d \int_\YY f_y\,\d \mu(y)=\int_\YY \d f_y\,\d \mu(y)
\] 
(resp.\   $\Delta \int_\YY f_y\,\d \mu(y)=\int_\YY \Delta f_y\,\d \mu(y)$, resp.\ $\nabla^*\int_\YY A_y\,\d \mu(y)=\int_\YY \nabla^*A_y\,\d \mu(y)$).
\end{proposition}

\begin{proof} We start with the case of differential.  We have already noticed that $\d$ is weakly local and we know from \cite[Theorem 2.2.9]{Gigli14} that it is a closed operator. Let us check that the other assumptions in Lemma \ref{le:hille} are satisfied. $(i)$ holds by our assumption, thus we pass to  $(ii)$. Let $U,V$ as in the statement and let $\phi\in\Lip(\X,\dist)$ be identically 1 on $V$ and with support in $U$ (the hypothesis $\sfd(V,\X\setminus U)>0$ grants that such $\phi$ exists). We define $T(f):=\phi f$ and notice that the properties in \eqref{eq:cutoffT} are trivial. We pass to $(iii)$, and notice that by the very definition of extension of $\d$ from $H^{1,2}(\X,\sfd,\mm)$ to $H^{1,2}_{loc}(\X,\sfd,\mm)$ it follows that the Leibniz rule holds even in $H^{1,2}_{loc}(\X,\sfd,\mm)$. It is then clear that we have $\phi f\in H^{1,2}_{loc}(\X,\sfd,\mm)$ with  $\d (\phi f)=\phi \d f+f\d \phi$. Thus 
\begin{equation}
\label{eq:daleibd}
|\d (\phi f)|\leq  \chi_{U}|\d f|\sup|\phi|+\chi_U C |f| \in L^2(\X,\mm),
\end{equation}
where $C$ denotes the Lipschitz constant of $\phi$. 
Therefore $\phi f$ is actually in $H^{1,2}(\XX,\dist,\mass)$.

With this said, let us verify the first claim. Fix $U\subset \X$ open bounded, let $\phi\in\Lip_{bs}(\X,\sfd)$ be identically 1 on $U$ and notice that replacing $f_y$ with $\phi f_y$ it is sufficient to prove that if $y\mapsto f_y\in L^2(\X,\mm)$ is  Borel and $f_y\in H^{1,2}(\X,\sfd,\mm)$ for every $y\in\Y$, then $y\mapsto \d f_y\in L^2(T^*(\X,\sfd,\mm))$ is strongly Borel. Since   $L^2(T^*(\X,\sfd,\mm))$ is separable (see \cite{ACM14} and \cite[Proposition 2.2.5]{Gigli14}), it is enough to check Borel regularity. Also, since $\d:H^{1,2}(\X,\dist,\meas)\to L^2(T^*(\X,\dist,\meas))$ is continuous, it suffices to prove that $y\mapsto  f_y\in H^{1,2}(\X,\sfd,\mm)$ is Borel. To see this it is sufficient to show that the unit ball in $H^{1,2}(\X,\sfd,\mm)$ belongs to the $\sigma$-algebra $\mathcal A$ generated by $L^2$-open sets in $H^{1,2}(\X,\sfd,\mm)$. But this is obvious, because the lower semicontinuity of the $H^{1,2}$-norm w.r.t.\ $L^2$-convergence ensures that closed $H^{1,2}$-balls are also $L^2$-closed, and thus are in $\mathcal A$. Since open balls are countable unions of closed balls, the first claim  follows. 

For the second we now observe that what just proved, our assumption $\int_\YY  \|\chi_U|\d f_y|\|_{L^2}\,\d\mu(y)<\infty$  and \eqref{eq:daleibd} ensure that  $y\mapsto \d(\phi f_y)$ is in $L^1(\Y,\mu;L^2(T^*(\X,\sfd,\mm)))$, i.e.\ $(iii)$ of Lemma \ref{le:hille} holds and the conclusion follows from such lemma.

The same line of thought gives the conclusion for $\nabla^*$. For the Laplacian we start noticing that  for $V,U$ and $\phi$ as above we have
\[
\begin{split}
\int_\XX|\phi|^2|\d f|^2\,\d\mm&=-\int_\XX2 f\phi\langle\d f,\d\phi\rangle+\phi^2 f\Delta f\,\d\mm\\
&\leq \int_\XX\tfrac12 |\phi|^2|\d f|^2+2|\d\phi|^2|f|^2+\tfrac12|\phi|^2(|f|^2+|\Delta f|^2)\,\d\mm,
\end{split}
\]
i.e.\ $\frac12\int_\XX|\phi|^2|\d f|^2\,\d\mm\leq C\int_U |f|^2+|\Delta f|^2\,\d\mm $. This proves that if $f\in D_{loc}(\Delta)\subset L^2_{loc}(\X,\mm)$ then $f\in H^{1,2}_{loc}(\XX,\dist,\mass)$ as well. Hence what previously proved tells that for  $y\mapsto f_y\in L^2_{loc}$ Borel with $f_y\in D_{loc}(\Delta)$ for $\mu$-a.e.\ $y$, we have that $y\mapsto \chi_U\d f_y\in L^2(T^*(\X,\sfd,\mm))$ is strongly Borel  for any $U\subset\X$ open bounded. Now we want to prove that the same assumptions ensure that $y\mapsto \chi_U\Delta f_y\in L^2(\X,\mm)$ is Borel as well. Since the $\sigma$-algebra generated by the strong topology coincides with that generated by the weak topology (because the closed unit ball can be realized as countable intersection of weakly-closed halfspaces, so that closed balls are weakly Borel and thus the same holds for open balls since they are countable union of closed balls), by approximation to get the desired Borel regularity it is sufficient to prove that $y\mapsto \int_\XX \psi \xi\Delta f_y \,\d\mm$ is Borel for any $\psi\in\Lip_{bs}(\X,\sfd)$ and $\xi$ varying in a countable dense subset of $L^2(\X,\mm)$. We pick $\xi\in H^{1,2}(\X,\sfd,\mm)$ and notice that
\[
\int_\XX \psi \xi\Delta f_y \,\d\mm=-\int_\XX \langle \nabla(\psi \xi),\nabla f_y\rangle \,\d\mm=-\int_U \langle \nabla(\psi \xi),\nabla f_y\rangle \,\d\mm
\]
for any $U\subset \X$ open bounded and containing the support of $\psi$. 
By what we already proved we see that the RHS is a Borel function of $y$, hence the desired Borel regularity follows.

With this said, the conclusion for the Laplacian follows by first applying the result to the differential and then to the divergence (the study of the divergence operator closely follows that of $\nabla^*$ that in turn, as said, is largely based on that of $\d$).
\end{proof}
\begin{remark}
The above version of Hille's theorem is compatible with  the more general one recently discussed in \cite[Section 3.3]{CGP21}. 
However, as the presentation here is substantially simpler we preferred giving a direct proof, rather than linking the terminology to that in \cite{CGP21}.
\fr\end{remark}

\subsection{Gaussian estimates and their consequences}
Thanks to (\ref{eq:poincare}) and (\ref{eq:bishgrom}), it follows from \cite[Proposition 3.1]{Sturm96III} that there exists a unique (locally H\"older) continuous function $p: \XX \times \XX \times (0, \infty) \to (0, \infty)$, called the \textit{heat kernel} of $(\XX, \dist, \meas)$, such that the following holds;
\begin{equation}\label{eq:heatflowheatkernel}
\heat_tf(x)=\int_{\XX}p(x, y, t)f(y)\di \meas(y)\qquad \forall f \in L^2(\XX, \meas),\forall x \in \XX.
\end{equation}
Let us denote by $p_{y, t}(x)=p(x, y, t)$ when we consider $p$ as a function on $\XX$ for fixed $y \in \XX$ and $t>0$. 

Let us recall the Gaussian estimates for the heat kernel $p$ proved in \cite{jiang2014heat}, where we are going to use them only specialized to the case $\epsilon=1$: for any $\epsilon \in (0, 1]$ there exists a positive constant $C:=C(K, N, \epsilon)$ depending only on $K, N$ and $\epsilon$ such that for any $x, y \in \XX$ and any $0<t<1$,
\begin{equation}\label{eq:gaussian}
\frac{C}{\meas (B_{\sqrt{t}}(x))}\exp \left(-\frac{\dist (x, y)^2}{(4-\epsilon)t}-Ct \right) \le p(x, y, t) \le \frac{C}{\meas (B_{\sqrt{t}}(x))}\exp \left( -\frac{\dist (x, y)^2}{(4+\epsilon)t}+Ct \right),
\end{equation}
and for every $y\in\X$ and $t>0$ we have 
\begin{equation}\label{eq:equi lip}
|\d p_{y, t}|(x)\le \frac{C}{\sqrt{t}\meas (B_{\sqrt{t}}(x))}\exp \left(-\frac{\dist(x, y)^2}{(4+\epsilon) t}+Ct\right)\qquad\mm\text{-a.e.\ }x\in\X.
\end{equation}
Notice that \eqref{eq:gaussian} and Lemma \ref{lem:volume} ensure that $p(\cdot,y,t)\in L^2(\X,\mm)$ for every $y\in\X$, $t>0$, therefore from \eqref{eq:heatflowheatkernel} we deduce the
 Chapman-Kolmogorov equation:
\begin{equation}\label{eq:ChapKol}
p(x,y,t+s)=\heat_{t} p(\,\cdot\,,y,s) (x)=\int_\XX p(x,z,t) p (z,y,s)\dd{\mass(z)}\quad \forall t, s>0, \forall x, y \in \XX.
\end{equation}
Also, from   \eqref{eq:heatfl}, \cite[Corollary 2.7]{Sturm96II} and \eqref{eq:gaussian}  we deduce the estimate
\begin{equation}\label{eq:lapheatbd}
\left|\Delta p(\cdot, y, t)\right|(x) \le \frac{C}{t\meas (B_{\sqrt{t}}(x))}\exp \left( -\frac{\dist (x, y)^2}{(4+\epsilon)t}+Ct\right)\qquad\mm\text{-a.e.\ }x\in\X,
\end{equation}
for every $y\in\X$, $t>0$. Notice that the above discussion and estimates easily imply that
\[
p_{y,t}\in\mathrm{TestF}(\XX,\dist,\meas)
\] 
for every $y\in\X$, and $t>0$. We shall frequently use this fact. For future reference we also notice that \eqref{eq:ChapKol} and the estimate \eqref{eq:gaussian} together with \eqref{eq:classheat1}, \eqref{eq:bishgrom} and Lemma \ref{lem:volume} give
\begin{equation}
\label{eq:stimah12}
\|p_{y,t}\|_{H^{1,2}}+\|\Delta p_{y,t}\|_{H^{1,2}}\leq C(K,N,t)\mm(B_{\sqrt t}(y))^{-\frac12}.
\end{equation}
We also notice that the identity $\partial_tp(x,y,t)=\partial_tp_{y,t}(x)=\Delta p_{y,t}(x)=\Delta_xp(x,y,t)$ valid for any $t>0$, $y\in\X$ and a.e.\ $x$ together with the symmetry in $x,y$ of the heat kernel - and thus of the LHS - gives
\begin{equation}
\label{eq:scambioLap}
\Delta_x p(x,y,t)=\Delta_y p(x,y,t)\qquad (\mm\times\mm)\text{-a.e.\ }(x,y),\ \forall t>0.
\end{equation}
We conclude pointing out that the continuity of the heat kernel and the estimates \eqref{eq:gaussian} ensure that for any $t>0$ the map $y\mapsto p_{y,t}\in L^2(\X,\mm)$ is continuous. Thus by the first claim in Proposition \ref{prop:lochille} we deduce that $y\mapsto \d p_{y,t}\in L^2(T^*(\X,\sfd,\mm))$ is strongly Borel. Similarly for $y\mapsto \Delta p_{y,t}$ and $y\mapsto \d\Delta p_{y,t}$.

\subsection{Smoothing metrics $g_t$ and computation of $\nabla^*g_t$}
In order to introduce the main tool in this paper, i.e.\ the \textit{smoothing metrics} $g_t$, let us start this subsection by observing the smooth case as follows.

For an $n$-dimensional weighted complete Riemannian manifold $(M, g, e^{-V}\dd \mathrm{Vol}_g)$ satisfying ${\rm Ric}_N\geq Kg$ for some $K \in \mathbb{R}$ and some $N \in [n, \infty)$ (namely $(M, \dist_g, e^{-V} \mathrm{Vol}_g)$ is an $\RCD(K, N)$ space, recall \eqref{eq:beric} and \eqref{eq:bekn}), for any $t>0$, define the map $\Phi_t:M \to L^2(M, e^{-V} \mathrm{Vol}_g)$ by 
\[
\Phi_t(x):=(y \mapsto p(x, y, t)) \in L^2(M, e^{-V} \mathrm{Vol}_g).
\]
Then the pull-back $g_t:=(\Phi_t)^*g_{L^2}$ is well-defined as a smooth tensor of type $(0, 2)$ and it satisfies
\[
g_t(x)=\int_M \d_xp_{y, t}(x)\otimes\d _x p_{y, t}(x) e^{-V(y)}\dd \mathrm{Vol}_g(y)\qquad \forall x \in M
\]
where it is emphasized that the RHS of the above makes sense as Bochner integral for \textit{any} $x \in M$ because of (\ref{eq:equi lip}).  
In particular, thanks to Fubini's theorem for all smooth vector fields $V_i\ (i=1, 2)$ on $M$ with bounded supports we have 
\[
\int_M g_t(V_1, V_2)e^{-V}\dd \mathrm{Vol}_g =\int_M\int_M \d_xp_{y, t} (V_1)(x) \,\d_x p_{y, t}( V_2)(x)e^{-V(x)-V(y)}\dd \mathrm{Vol}_g(x)\dd \mathrm{Vol}_g(y)
\]
and it is easy to see that this equation also characterizes $g_t$.

Let us generalize this observation to an $\RCD(K, N)$ space $(\XX, \dist, \meas)$ as follows. Start noticing that the identity $|\d p_{y,t}\otimes \d p_{y,t}|_{\sf HS}=|\d p_{y,t}|^2$, \eqref{eq:stimelinftyd} below and Lemma \ref{lem:volume} ensure that for every $t>0$ the map $y\mapsto \d p_{y,t}\otimes \d p_{y,t}$ is in $L^1(\X,\mm; L^2_{loc}((T^*)^{\otimes 2}(\X,\sfd,\mm)))$, namely for a bounded $A \subset \XX$ and a fixed $\bar x \in A$,
\[
\begin{split}
\int_\XX\sqrt{\int_A  |\d p_{y,t}|^2\,\d\mm(x)}\,\d\mm(y)\leq 
  C \sqrt{\meas (A)}\int_\XX e^{-\frac{\sfd^2(\bar x, y)}{5 t}}\,\d\mm(y)<\infty.
\end{split}
\]
 Hence the following definition is well-posed:

\begin{definition}[Smoothing metrics $g_t$]\label{def:gt} Let $(\X,\sfd,\mm)$ be an $\RCD(K,N)$ space. We define the $(0, 2)$ tensor $g_t\in L^0((T^*)^{\otimes 2}(\X,\sfd,\mm))$ on $\XX$ as
\[
g_t:=\int_\XX\d_xp_{y, t}\otimes\d _x p_{y, t} \,\d\mm(y).
\]
\end{definition}

Notice that the basic properties of Bochner integration (Hille's Theorem) ensure that for $V_1,V_2\in L^2(T(\X, \dist, \meas))$ with bounded support we have
\[
\int_\XX g_t(V_1, V_2) \,\d\mm= \int_\XX\int_\XX \d p_{y, t} (V_1)\,\d p_{y, t}(V_2)  \di \meas \di \meas(y).
\]
After a normalization of $g_t$ as follows, the smoothing metrics are uniformly bounded  in $L^{\infty}$:
\begin{proposition} Let $(\X,\sfd,\mm)$ be an $\RCD(K,N)$ space. Then we have 
\begin{equation}
\label{eq:Linfgt}
    t\meas(B_{\sqrt{t}}(\,\cdot\,))g_t\le C(K,N) g\qquad\mm\text{-a.e.,\ } \forall t \in (0, 1]
\end{equation}
in the sense of symmetric tensors.
In particular we have $g_t \in L_{{loc}}^{\infty}((T^*)^{\otimes 2}(\XX, \dist, \meas))$ and moreover $t\meas(B_{\sqrt{t}}( \cdot ))g_t \in L^{\infty}((T^*)^{\otimes 2}(\XX, \dist, \meas))$.
\end{proposition}

\begin{proof}
Fix $V\in L^0(T(\X,\sfd,\mm))$ and notice that for $\mm$-a.e.\ $x$ we have
\[
\begin{split}
t\mm(B_{\sqrt t}(x))g_t(V,V)(x)\leq t\mm(B_{\sqrt t}(x))|V|^2(x)\int_\XX |\d p_{y,t}|^2(x)\,\d\mm(y)\\
\text{(by \eqref{eq:equi lip})}\qquad\leq \frac{C|V|^2(x)}{\mm(B_{\sqrt t}(x))}\int_\XX \exp \left(-\frac{\dist(x, y)^2}{5 t}+Ct\right)\,\d\mm(y).
\end{split}
\]
The conclusion follows from  Lemma \ref{lem:volume} (with $\alpha=0$).
\end{proof}
We now turn to the computation of $\nabla^*g_t$. To this aim, it is convenient to introduce the following function:
\[
    p_t(x)\defeq p(x,x,t)\stackrel{\eqref{eq:ChapKol}}=\int_\X p_{y,t/2}^2(x)\,\d\mm(y).
\]
Notice that thanks to the bounds \eqref{eq:gaussian} it is easy to see that for every $t>0$ the map $y\mapsto p_{y,t/2}^2$ is in $L^1(\X,\mm;L^2(\X,\mm))$. It is then clear that the identity $p_t=\int_\XX p_{y,t/2}^2\,\d\mm(y)$ holds also in the sense of Bochner integrals.

 Let us start collecting some estimates for this function:
 
\begin{lemma}\label{lem:diagonalformula} Let $(\X,\sfd,\mm)$ be an $\RCD(K,N)$ space. 

Then for any $t >0$ we have $p_{2t}(x) \in  D_{loc}(\Delta)$ with 
\begin{equation}\label{eq:dptbound}
\d p_{t}=2\int_\X p_{y,t/2}\d p_{y,t/2}\,\d\mm(y)\qquad\text{and}\qquad|\d  p_{t}|\le \frac{C(K, N)}{\sqrt{t}\meas (B_{\sqrt{t}}(\cdot))}\qquad \mm\text{-a.e.}
\end{equation}
and
\begin{equation}\label{eq:laplacianbound}
\Delta p_{t}=2\int_\X p_{y,t/2}\Delta p_{y,t/2}+|\d p_{y,t}|^2\,\d\mm(y)\qquad\text{and}\qquad|\Delta p_{t}| \le \frac{C(K, N)}{t\meas (B_{\sqrt{t}}(\cdot))} \qquad \meas\text{-a.e..}
\end{equation}
Finally, we also have $\Delta p_{t}\in H^{1,2}_{loc}(\X,\sfd,\mm)$ with
\begin{equation}
\label{eq:ddeltap2t}
\d\Delta p_{t}=2\int_\XX  \d p_{y,t/2}\Delta p_{y,t/2}+p_{y,t/2}\d \Delta p_{y,t/2}+\d|\d p_{y,t}|^2\,\d\mm(y).
\end{equation}
It is part of the claim the fact that the integrands in \eqref{eq:dptbound} and \eqref{eq:ddeltap2t} belong to the space $L^1(\X,\mm;L^2_{loc}(T^*(\X,\sfd,\mm)))$ and the one in \eqref{eq:laplacianbound} belongs to the space $L^1(\X,\mm;L^2_{loc}(\X,\mm))$.
\end{lemma}

\begin{proof} Using \eqref{eq:gaussian} and \eqref{eq:equi lip}  we get
\[
\begin{split}
\int_\XX |\d ( p_{y,t/2}^2)|\,\d\mm(y)&\leq 2\int_\XX p_{y,t/2}|\d p_{y,t/2}|\,\d\mm(y)\leq\frac{C}{\sqrt t\mm(B_{\sqrt t}(\cdot))^2}\int_\XX \exp(-\frac {2\sfd^2(\cdot,y)}{5t}+Ct)\,\d\mm(y).
\end{split}
\]
Thus from Lemma \ref{lem:volume} and Proposition \ref{prop:lochille} we deduce that 
$p_{t}\in H^{1,2}_{loc}(\X,\sfd,\mm)$ and that \eqref{eq:dptbound} holds. Similarly,  starting from 
\[
\begin{split}
\int_\XX| \Delta (p_{y,t/2}^2)|\,\d\mm(y)\leq2\int_\XX p_{y,t/2}|\Delta p_{y,t/2}|+|\d p_{y,t/2}|^2\,\d\mm(y)
\end{split}
\]
and using the estimates \eqref{eq:gaussian}, \eqref{eq:equi lip} and \eqref{eq:lapheatbd}  and then again Lemma \ref{lem:volume} and Proposition \ref{prop:lochille}, we conclude that 
$p_{t}(x) \in  D_{{loc}}(\Delta)$ and that \eqref{eq:laplacianbound} holds.

For the last claim we recall that  $p_{y,t}\in \mathrm{TestF}(\XX,\dist,\meas)$, thus the Leibniz rule for the Laplacian and the basic properties of test functions give $p_{y,t/2}\Delta p_{y,t/2}+|\d p_{y,t}|^2\in H^{1,2}(\X,\sfd,\mm)$ with 
\[
\d(p_{y,t/2}\Delta p_{y,t/2}+|\d p_{y,t}|^2)=\d p_{y,t/2}\Delta p_{y,t/2}+p_{y,t/2}\d\Delta p_{y,t/2}+2{\rm Hess}\,p_{y,t}(\nabla p_{y,t},\,\cdot\,).
\] 
The fact that the first two addends in the RHS are in $L^1(\X,\mm;L^2_{loc}(T^*(\X,\sfd,\mm)))$ can be proved as before. For the last one we let $A\subset \X$ be Borel and bounded and $\bar x\in \X$. Then we have $\sfd(x,y)\geq \sfd(y,\bar x)-R$ for any $x\in A$, $y\in\X$ and some $R>0$ independent on $x,y$. Hence \eqref{eq:equi lip} implies that $\||\d p_{y,t}|\|_{L^\infty(A)}\leq Ce^{-\frac{\sfd^2(y,\bar x)}{5t}}$ for some $C=C(t,K,N,A,\bar x)$, thus
\[
\begin{split}
\int_\XX\sqrt{\int_A  |{\rm Hess}\,p_{y,t}(\nabla p_{y,t},\,\cdot\,)|^2\,\d\mm(x)}\,\d\mm(y)&\leq \int_\XX \||{\rm Hess}\,p_{y,t}|_{\sf HS}\|_{L^2}\| |\d p_{y,t}|\|_{L^\infty(A)}\,\d\mm(y)\\
\textrm{(by \eqref{eq:bochner2})}\qquad&\leq  C \int_\XX (\|\Delta p_{y,t}\|_{L^2} +\||\d p_{y,t}|\|_{L^2})e^{-\frac{\sfd^2(y,\bar x)}{5t}}\,\d\mm(y)\\
\textrm{(by \eqref{eq:stimah12})}\qquad&\leq  C \int_\XX \mm(B_{\sqrt t}(y))^{-\frac12}e^{-\frac{\sfd^2(y,\bar x)}{5t}}\,\d\mm(y)<\infty,
\end{split}
\]
where the last inequality comes from Lemma \ref{lem:volume}. The conclusion follows.
\end{proof}
To further analyze the link between $g_t$ and $p_t$ the following result will be crucial: 
\begin{lemma}\label{lem:MoveLap2}
Let $(\X,\sfd,\mm)$ be an $\RCD(K,N)$ space. Then for every $t>0$ we have that $y\mapsto \Delta p_{y,t}\d p_{y,t}$ and $y\mapsto  p_{y,t}\d\Delta p_{y,t}$ are both in $L^1(\X,\mm;L^2_{loc}(T^*(\X,\sfd,\mm)))$ and
\[
\int_\XX \Delta p_{y,t}\d p_{y,t}\,\d\mm(y)=\int_\XX p_{y,t}\d\Delta p_{y,t}\,\d\mm(y).
\]
\end{lemma}
\begin{proof} For the first part of the claim we argue as in the proof of Lemma \ref{lem:diagonalformula} above: let $A\subset \X$ be Borel and bounded and $\bar x\in \X$.  
Then, \eqref{eq:equi lip} implies that for some $C=C(t,K,N,A,\bar x)$
\begin{equation}
\label{eq:stimelinftyd}
\||\d p_{y,t}|\|_{L^\infty(A)}\leq\Big \Vert Ce^{-\frac{\sfd^2(y, \,\cdot\,)}{{(4+1/2)}t}}\Big\Vert_{L^\infty(A)}\leq Ce^{-\frac{\sfd^2(y,\bar x)}{5t}},
\end{equation}
so that,
\[
\int_\XX \sqrt{\int_A| \Delta p_{y,t}\d p_{y,t}|^2\,\d\mm(x) }\,\d\mm(y)\leq C\int_\XX \|\Delta p_{y,t}\|_{L^2}e^{-\frac{\sfd^2(y,\bar x)}{5t}}\,\d\mm(y)<\infty,
\]
having used the bound \eqref{eq:stimah12} and Lemma \ref{lem:volume} in the last step. This proves that  $y\mapsto \Delta p_{y,t}\d p_{y,t}$  is  in $L^1(\X,\mm;L^2_{loc}(T^*(\X,\sfd,\mm)))$ and an analogous argument gives the same for $y\mapsto  p_{y,t}\d\Delta p_{y,t}$.

Now write the Chapman-Kolmogorov equation \eqref{eq:ChapKol} as
\[
\int_\XX p(y,z,s) p_{z,t}\,\d \mm(z)=p_{y,t+s}
\]
and observe that the estimates \eqref{eq:gaussian}, \eqref{eq:equi lip} and the same arguments just used ensure that for any $y\in\X$ the maps $z\mapsto p(y,z,s) p_{z,t}$ and $z\mapsto p(y,z,s) \d p_{z,t}$ are in $L^1(\X,\mm;L^2_{loc}(\X,\mm))$ and $L^1(\X,\mm;L^2_{loc}(T^*(\X,\sfd,\mm)))$ respectively. Thus Proposition \ref{prop:lochille} gives
\[
\int_\XX p(y,z,s)\d p_{z,t}\,\d \mm(z)=\d p_{y,t+s}.
\]
Multiplying both sides by $p_{y,t}$, integrating in $y$ and using Fubini's theorem we obtain
\[
\int_\XX p_{z,t+s}\d p_{z,t}\,\d\mm(z)\stackrel{\eqref{eq:ChapKol}}=\int_\XX\int_\XX p_{y,t}\,p(y,z,s)\d p_{z,t}\,\d \mm(z)\,\d\mm(y)=\int_\XX p_{y,t}\d p_{y,t+s}\,\d\mm(y).
\]
Thus to conclude it is sufficient to prove that as $s\to 0^+$ we have
\begin{equation}
\label{eq:togetml}
\begin{split}
\int_\XX \frac{p_{y,t+s}-p_{y,t}}s\d p_{y,t}\,\d\mm(y)&\to \int_\XX\Delta p_{y,t}\d p_{y,t}\,\d\mm(y),\\
\int_\XX p_{y,t}\d \Big(\frac{p_{y,t+s}-p_{y,t}}s\Big)\,\d\mm(y)&\to \int_\XX  p_{y,t}\d \Delta p_{y,t}\,\d\mm(y)
\end{split}
\end{equation}
in $L^2_{loc}(T^*(\X,\sfd,\mm))$. We start noticing that from \eqref{eq:ChapKol} we have
\[
\begin{split}
F_{y,t}:= \frac{p_{y,t+s}-p_{y,t}}s-\Delta p_{y,t}&=\int_0^1 \Delta(p_{y,t+rs}-p_{y,t})\,\d r\\
&=s\int_0^1 r\Delta\Big(\int_0^1\Delta p_{y,t+rsh}\,\d h\Big)\,\d r\\
&=s\int_0^1 r\Delta \heat_{t/3}\Big(\int_0^1\Delta \heat_{t/3} p_{y,t/3+rsh}\,\d h\Big)\,\d r
\end{split}
\]
and therefore using twice \eqref{eq:classheat1} we obtain 
\begin{equation}
\label{eq:Fyt}
\|F_{y,t}\|_{L^2}\leq sC(t)\int_0^1 \int_0^1\|p_{y,t/3+rsh}\|_{L^2}\,\d h\,\d r\stackrel{\eqref{eq:l2decr}}\leq  sC(t)\|p_{y,t/3}\|_{L^2}\stackrel{\eqref{eq:stimah12}}\leq s \,C(K,N,t)\mm(B_{\sqrt{\frac t3} }(y))^{-\frac12}.
\end{equation}
Thus for $A\subset\X$ Borel and bounded we have
\[
\int_\XX \sqrt{\int_A |F_{y,t}\d p_{y,t}|^2\,\d\mm}\,\d\mm(y)\stackrel{\eqref{eq:stimelinftyd}}\leq C \int_\XX \|F_{y,t}\|_{L^2}  e^{-\frac{\sfd^2(y,\bar x)}{5t}}\,\d\mm(y)\stackrel{\eqref{eq:Fyt}}\leq   s\,C\int_\XX  \frac{e^{-\frac{\sfd^2(y,\bar x)}{5t}}}{\mm(B_{\sqrt{\frac t3} }(y))^{\frac12}}\,\d\mm(y)
\]
for some $C=C(K,N,t,A,\bar x)$. Since the last integral is finite by Lemma \ref{lem:volume}, the LHS goes to 0 as $s\to 0^+$. This proves the first in \eqref{eq:togetml}. The second follows along very similar lines, we omit the details.
\end{proof}
We are now in a position to prove the main result of this subsection:
\begin{theorem}\label{thm:DiffKeyFormula2}
Let $(\X,\sfd,\mm)$ be an $\RCD(K,N)$ space. Then for every $t>0$ we have  $g_t \in D_{{loc}}(\nabla^*)$ with 
 \begin{equation}\label{eq:KeyDiff2}
     \nabla^* g_t=-\frac{1}{4}\dd \Delta p_{2t}.
 \end{equation}
\end{theorem}
\begin{proof} For any $y\in\X$ Proposition \ref{prop:adjoing formula} tells   $\nabla^*(\dd p_{y, t}\otimes \dd p_{y, t})=-\Delta p_{y, t}\dd  p_{y, t}-\frac{1}{2}\dd |\nabla p_{y, t}|^2$. Also, arguing as in Lemma \ref{lem:diagonalformula} it is easy to see that $y\mapsto -\Delta p_{y, t}\dd  p_{y, t}-\frac{1}{2}\dd |\nabla p_{y, t}|^2 $ belongs to the space $L^1(\X,\mm;L^2_{loc}(T^*(\X,\sfd,\mm)))$. Thus taking into account Lemma \ref{lem:MoveLap2} we obtain
\[
\begin{split}
\int_\XX \nabla^*(\dd p_{y, t}\otimes \dd p_{y, t})\,\d\mm(y)&=-\tfrac12\int_\XX \Delta p_{y, t}\dd  p_{y, t}+ p_{y, t}\dd  \Delta p_{y, t}+\dd |\d p_{y, t}|^2\,\d\mm(y)\stackrel{\eqref{eq:ddeltap2t}}=-\frac{1}{4}\dd \Delta p_{2t}.
\end{split}
\]
The conclusion  comes from the very definition of $g_t$ and Proposition \ref{prop:lochille}.
\end{proof}


\subsection{Asymptotic behaviour as $t\to 0^+$}
The goal of this subsection is to study the behaviour of $g_t$ and $\nabla^*g_t$ as $t\to0^+$.

We start with the following result, which generalizes to the non-compact setting the analogous statement \cite[Theorem 5.10]{AHPT}:

\begin{theorem}\label{thm:pullbackconvergence} Let $(\X,\sfd,\mm)$ be an $\RCD(K,N)$ space of essential dimension $n$. 

Then
$t\meas(B_{\sqrt{t}}(\,\cdot\,))g_t\to c_n g$ strongly in $L^p_{{loc}}$ for any $p\in [1,\infty)$, where $c_n$ is a positive constant depending only on $n$.
\end{theorem}

\begin{proof}
Since the proof is essentially same to that in \cite[Theorem 5.10]{AHPT} after replacing $L^p$ by $L^p_{{loc}}$ (recall that \cite[Theorem 5.10]{AHPT} discussed only on the case when $(\XX, \dist)$ is compact), we shall only give a sketch of the proof.

Fix $V\in L^{\infty}(T(\XX, \dist, \meas))$ with bounded support.
First let us discuss the asymptotic behaviour of the following as $t \to 0^+$ for fixed $y \in \XX$ and $L>0$;
\begin{equation}\label{eq:separate}
\begin{split}
	&\int_\XX t\meas (B_{\sqrt{t}}(x))|\d p_{y, t}(V)|^2(x) \di \meas(x)  \\
&=\int_{B_{L\sqrt{t}}(y)}t\meas (B_{\sqrt{t}}(\cdot))|\d p_{y, t}(V)|^2\di \meas + \int_{\XX\setminus B_{L\sqrt{t}}(y)}t\meas (B_{\sqrt{t}}(\cdot))|\d p_{y, t}(V)|^2\di \meas.
\end{split}
\end{equation}
The key idea to control the each terms in the RHS of (\ref{eq:separate}) is to apply \textit{blow-up arguments} (i.e.\ we discuss the behaviour of the rescaled spaces $(\XX, \sqrt{t}^{-1}\dist, \meas(B_{\sqrt{t}}(y))^{-1}\meas, y)$ with respect to the pointed measured Gromov-Hausdorff convergence as $t \to 0^+$) in conjunction with the stability of the heat flow first observed in  \cite{Gigli10}. More precisely, we use the stability results proved in \cite[Corollary 5.5, Theorem 5.7, Lemma 5.8]{ambrosio2017local}, \cite[Theorem 4.4]{AH16}, \cite[Theorem 3.3]{AHT17} (with \cite[Theorem 2.19]{AHPT}), \cite[Theorem 6.8]{GMS15} with Theorem \ref{thm:RN} and (\ref{eq:equi lip}). Combining these, letting  $t \to 0^+$ and then letting $L \to \infty$ in the RHS of (\ref{eq:separate}), the following hold for $\meas$-a.e.\ $y \in \XX$:
\begin{enumerate}
\item The first term of the RHS of (\ref{eq:separate}) converges to $c_n|V|^2(y)$.
\item The second term of the RHS of (\ref{eq:separate}) converges to $0$.
\end{enumerate}
Thus as $t \to 0^+$ we obtain
\[
\int_\XX t\meas (B_{\sqrt{t}}(x))|\d p_{y, t}(V)|^2(x)\di \meas(x) \to c_n|V|^2(y)\qquad \text{$\meas$-a.e.\ $y \in \XX$.}
\]
Thus combining this with (\ref{eq:Linfgt}) and the dominated convergence theorem we get
\[
\int_{\XX}t\meas (B_{\sqrt{t}}( \cdot))g_t(V, V)\di \meas \to c_n\int_{\XX}|V|^2\di \meas
\]
which proves that $t\meas(B_{\sqrt{t}}(\,\cdot\,))g_t$ $L^p$-weakly converge to $c_n g$ on any bounded subset $A$ of $\XX$ because $g_t$ is symmetric and $V$ is arbitrary.

In order to get the $L^p_{loc}$-strong convergence it suffices to check
\begin{equation}\label{eq:limit4}
\lim_{t \to 0^+}\int_\XX \phi |t\meas (B_{\sqrt{t}}( \cdot))g_t|_{\mathrm{HS}}^2\di \meas =c_n^2\int_\XX \phi |g|_{\mathrm{HS}}^2\di \meas=c_n^2n\int_\XX\,\phi\d\mm
\end{equation}
for every $\phi\in \Lip_{bs}(\X, \dist)$, because this implies the $L^2_{loc}$-strong convergence and the improvement to the $L^p_{loc}$-strong one comes from (\ref{eq:Linfgt}). 
Let us check (\ref{eq:limit4}) as follows.

For any $z \in \mathcal{R}_n$, applying blow-up arguments as explained above again allows us to deduce
\[
F(z, t):=\frac{1}{\meas (B_{\sqrt{t}}(z))}\int_{B_{\sqrt{t}}(z)}|t\meas (B_{\sqrt{t}}( \cdot))g_t|_{\mathrm{HS}}^2\di \meas \to c_n^2n
\]
and thus (recalling \eqref{eq:Linfgt} to use the dominate convergence theorem) for $\phi\in \Lip_{bs}(\X,\dist)$ we have
\begin{equation}
\label{eq:st1}
\lim_{t\to 0^+}\int_\XX\phi (z)F(z,t)\,\d\mm(z) = c_n^2n\int_\XX\phi\,\d\mm.
\end{equation}
On the other hand,  we have
\begin{equation}
\label{eq:st2}
\int_\XX\phi (z)F(z,t)\,\d\mm(z)=\int_\XX |t\meas (B_{\sqrt{t}}( \cdot))g_t|_{\mathrm{HS}}^2(x)\underbrace{\int_{B_{\sqrt t}(x)}\frac{\phi(z)}{\mm(B_{\sqrt t}(z))}\d\mm(z)}_{=:G(x,t)}\d\mm(x).
\end{equation}
Now notice that $\sup_{t,x}G(x,t)<\infty$ (because of \eqref{eq:boundmeasballs}) and $\lim_{t\to 0^+}G(x,t)= \phi(x)$ for $\mm$-a.e.\ $x$ (because of the convergence of the blow-ups to the Euclidean space). It follows (again using  \eqref{eq:Linfgt} to use the dominate convergence theorem) that
\[
\lim_{t\to 0^+}\int_\XX  |t\meas (B_{\sqrt{t}}(\cdot))g_t|_{\mathrm{HS}}^2(x)\Big|\int_{B_{\sqrt t}(x)}\frac{\phi(z)}{\mm(B_{\sqrt t}(z))}\d\mm(z)-\phi(x)\Big|\d\mm(x)=0
\]
which together with \eqref{eq:st1} and \eqref{eq:st2} gives \eqref{eq:limit4} and the conclusion.
\end{proof}

\begin{remark}\label{rem:singularity}
In Theorem \ref{thm:pullbackconvergence}, the conclusion can not be improved to the case when $p=\infty$ in general.
For example, the  $\RCD(0, 1)$ space $([0, \pi], \dist_{\mathbb{R}}, \haus^1)$ satisfies
\[
\liminf_{t \to 0^+}\left\|t\haus^1(B_{\sqrt{t}}(\cdot))g_t-c_1g_{\mathbb{R}}\right\|_{L^{\infty}}>0.
\]
See \cite[Remark 5.11]{AHPT} for details. It is worth pointing out that the verification of 
\begin{equation}\label{eq:loclinfty}
\|t\meas (B_{\sqrt{t}}(\cdot))g_t -c_n g\|_{L^{\infty}_{{loc}}} \to 0
\end{equation}
is closely related to the nonexistence of singular points (actually the singular points are $\{0, \pi\}$ in this example). See also \cite[Theorem 1.1]{honda2021sobolev}.

In connection with this pointing out, if $(M, g, e^{-V}\dd \mathrm{Vol}_g)$ is any weighted complete $n$-dimensional Riemannian manifold with ${\rm Ric}_N\geq Kg$ for some $K \in \mathbb{R}$ and some $N \in [n, \infty)$, applying a construction of the heat kernel by \textit{parametrix}, we can actually prove that (\ref{eq:loclinfty}) holds. 
More precisely we have as $t \to 0^+$.
\begin{equation}\notag
\begin{split}
	&4(8\pi)^{n/2}t^{(n+2)/2}g_{t}\\&=e^Vg-e^V\left(\frac{2}{3}\left(\mathrm{Ric}_g-\frac{1}{2}\mathrm{Scal}_g g\right)-\di V \otimes\di V -\Delta^gVg+\frac{|\nabla^g V|^2}{2}g\right)t +O(t^2),
\end{split}
\end{equation}
which is uniform on any bounded set.
See \cite[Theorem 5]{Berard:1994vy} and \cite[Theorem 3.5]{honda2020characterization} for  details. 
\fr\end{remark}
\begin{corollary}\label{cor:convgt} Let $(\X,\sfd,\mm)$ be an $\RCD(K,N)$ space of essential dimension $n$.  Let $A$ be a bounded Borel subset of $X$ with
\begin{equation}\label{eq:lowerbdRN}
\inf_{r \in (0, 1), x \in A}\frac{\meas (B_r(x))}{r^n}>0.
\end{equation}
Then  $\haus^n\mres A$ is a Radon measure absolutely continuous w.r.t.\ $\mm$ and
\[
\chi_A\omega_nt^{(n+2)/2}g_t \to\chi_A c_n\frac{\di \haus^n\res A}{\di \meas} g\qquad \text{in $L^p((T^*)^{\otimes 2}(\XX, \dist, \meas))$, $\forall p \in [1, \infty)$.}
\]
\end{corollary}

\begin{proof} The first part of the claim follows from Lemma \ref{le:densh} and \eqref{eq:lowerbdRN}. Then Theorem \ref{thm:RN} ensures that  as $r \to 0^+$
\[
\frac{\omega r^n}{\meas (B_r(x))} \to \frac{\di \haus^n\res A}{\di \meas}\qquad \text{ $\meas$-a.e.\ $x \in A$.}
\]
Thus \eqref{eq:lowerbdRN}, the dominated convergence theorem and Theorem \ref{thm:pullbackconvergence} give the conclusion.
\end{proof}

We now turn to the asymptotic of $\Delta p_{2t}$:

\begin{proposition}\label{prop:convlap} Let $(\X,\sfd,\mm)$ be an $\RCD(K,N)$ space, $K\in\R$, $N\in[1,\infty)$. Then as $t \to 0^+$ we have
\[
t\meas(B_{\sqrt{t}}(\cdot ))\Delta p_{2t}( \cdot) \to 0\qquad \text{ in $L^p_{{loc}}(\XX, \meas)$, $\forall p\in [1, \infty)$.}
\]
\end{proposition}

\begin{proof}
The proof is based on blow-up arguments which is similar to that of Theorem \ref{thm:pullbackconvergence}. Therefore we give only a sketch of the proof (see also \cite{AHPT}).

Let us first prove that for any $z \in \mathcal{R}_n$, as $t \to 0^+$,
\begin{equation}\label{eq:aibsh}
\frac{1}{\meas (B_{\sqrt{t}}(z))}\int_{B_{\sqrt{t}}(z)}t\meas(B_{\sqrt{t}}(x))|\Delta p_{2t}(x)|\di \meas(x) \to 0.  
\end{equation}
In order to prove this, consider the pointed measured Gromov-Hausdorff convergent sequence of the rescaled space:
\begin{equation}\label{eq:rescaled}
(\XX^{t, z}, \dist^{t, z}, \meas^{t, z}, z):=\left(\XX, \frac{1}{\sqrt{t}}\dist, \frac{1}{\meas(B_{\sqrt{t}}(z))}\meas, z\right) \stackrel{\mathrm{pmGH}}{\to} \left(\mathbb{R}^n,  \dist_{\mathbb{R}^n}, \frac{1}{\omega_n}\haus^n, 0_n\right)
\end{equation}
and denote by $p^{t, z}$, $\Delta^{t, z}$, the heat kernel, the Laplacian of $(\XX^{t, z}, \dist^{t, z}, \meas^{t, z})$, respectively, namely $p^{t, z}(x, y, s)=\meas(B_{\sqrt{t}}(z))p(x, y, ts)$, $\Delta^{t, z}f=t\Delta f$. Thus the LHS of  (\ref{eq:aibsh}) is equal to
\begin{equation}\label{eq:aosisjj}
\int_{B^{\dist^{t, z}}_1(z)}\meas^{t, z}(B_1^{\dist^{t, z}}(x))|\Delta^{t, z}p_2^{t, z}(x)|\di \meas^{t, z}(x).
\end{equation}
Applying the stability results already used in the proof of Theorem \ref{thm:pullbackconvergence} shows that $\Delta^{t, z}p_2^{t, z}$ $L^2_{{loc}}$-strongly converge to $\Delta^{g_{\mathbb{R}^n}}(\omega_n\tilde{p}_2)$ with respect to (\ref{eq:rescaled}), where $\tilde{p}(x)$ denotes the heat kernel of the $n$-dimensional Euclidean space evaluated at $(x,x)$. Since $\Delta^{g_{\mathbb{R}^n}}(\omega_n\tilde{p}_2)=0$ because $\tilde{p}_2$ is constant, (\ref{eq:aosisjj}) converges to 
\begin{equation}\notag
\int_{B_1(0_n)}|\Delta^{g_{\mathbb{R}^n}}(\omega_n\tilde{p}_2)|\di\left( \frac{1}{\omega_n}\haus^n\right)=0
\end{equation}
as $t \to 0^+$, which proves (\ref{eq:aibsh}).

Fix a $\phi\in \Lipbs(\XX,\dist)$. Applying (\ref{eq:aibsh}) with (\ref{eq:laplacianbound}), the dominated convergence theorem yields
\begin{equation}\notag
\int_{\XX}\frac{\phi(z)}{\meas (B_{\sqrt{t}}(z))}\int_{B_{\sqrt{t}}(z)}t\meas(B_{\sqrt{t}}(x))|\Delta p_{2t}(x)|\di \meas(x) \di \meas (z) \to 0.
\end{equation}
On the other hand (\ref{eq:laplacianbound}) and dominated convergence (recall \eqref{eq:boundmeasballs}) imply \begin{align}
&\left| \int_{\XX}\frac{\phi(z)}{\meas (B_{\sqrt{t}}(z))}\int_{B_{\sqrt{t}}(z)}t\meas(B_{\sqrt{t}}(x))|\Delta p_{2t}(x)|\di \meas(x) \di \meas (z)- \int_{\XX}\phi(z)t\meas(B_{\sqrt{t}}(z))|\Delta p_{2t}(z)|\di \meas(z)\right| \nonumber \\
&\le C(K, N) \int_\XX \Big| \phi(z)-\int_{B_{\sqrt{t}}(z)}\frac{\phi(x)}{\meas (B_{\sqrt{t}}(x))}\di \meas(x) \Big|\di \meas (z) \to 0.\notag
\end{align}
Thus
\begin{equation}\label{apsosirn}
\int_{\XX}\phi(x)t\meas (B_{\sqrt{t}}(x))|\Delta p_{2t}(x)|\di \meas (x) \to 0.
\end{equation}
The desired $L^p_{{loc}}$-strong convergence comes from (\ref{apsosirn}) and (\ref{eq:laplacianbound}).
\end{proof}
\begin{corollary}\label{cor:hhhss} Let $(\X,\sfd,\mm)$ be an $\RCD(K,N)$ space. Also, let $A$ be a bounded Borel subset of $X$ and $n\in\N$ be such that
\[
        \inf_{r \in (0, 1), x \in A}\frac{\meas(B_r(x))}{r^n}>0.
\]
    Then as $t \to 0^+$
\begin{equation}\notag
    t^{(n+2)/2} \Delta p_{2t}\to 0 \qquad \text{in $L^2(A,\mm)$}.
\end{equation}
\end{corollary}
\begin{proof}
Direct consequence of Proposition \ref{prop:convlap}.
\end{proof}
\begin{remark}
	Although the above convergence results are stated for the strong convergence in order to get our best knowledges, their weak convergences are enough to justify our main results as easily seen in the next section.
	\fr
\end{remark}

\section{Proof of the main results}
From both the technical and conceptual points of view, the following is the crucial result in this paper. Its proof is basically a combination of the convergence results established in Corollaries \ref{cor:convgt}, \ref{cor:hhhss} together with formula \eqref{eq:KeyDiff2}:

\begin{theorem}[Integration-by-parts formula]\label{intbyparts} Let $(\X,\sfd,\mm)$ be an $\RCD(K,N)$ space of essential dimension $n$. Let also $U\subset\X$ be open and assume that
\[
        \inf_{r \in (0, 1),x \in A}\frac{\meas(B_r(x))}{r^n}>0
\]
for every compact subset $A$ of $U$.  Then for any $\phi\in \Lip_{bs}(\X,\dist)$ with $\supp(\phi)\subset U$ and $f\in D(\Delta)$, it holds that 
\begin{equation}\notag
    \int_\XX \langle \dd\phi,\dd f\rangle\dd\haus^n =-\int_\XX \phi\tr(\hess f)\dd\haus^n.
\end{equation}
\end{theorem}

\begin{proof} The assumptions on $\phi,f$ ensure that $\phi\dd f$ is in the domain of the covariant derivative with $\nabla(\phi\dd f)=\dd\phi\otimes\dd f+\phi\hess f$ (see \cite[Theorem  3.4.2, Proposition 3.4.5]{Gigli14}), with identifications under the Riesz isomorphisms.  Thus \eqref{eq:KeyDiff2} gives
\begin{equation}\label{eq:KeyInt}
\begin{split}
	    \int_\XX\langle t^{(n+2)/2}g_t, \nabla(\phi\dd f)\rangle_{\sf HS}\,\dd\meas&=-\frac{1}{4}\int_\XX\langle \nabla \Delta (t^{(n+2)/2}p_{2t}), \phi \nabla f \rangle\dd\meas  \\
&=\frac{1}{4}\int_\XX\Delta (t^{(n+2)/2}p_{2t}) \mathrm{div}(\phi \nabla f) \dd\meas.
\end{split}
\end{equation}
Let us take the limit $t \to 0^+$ in (\ref{eq:KeyInt}). The RHS converge to $0$ because of Corollary \ref{cor:hhhss} applied with $A:=\supp(\phi)$.  On the other hand by Corollary \ref{cor:convgt} applied with $A:=\supp(\phi)$, the LHS of   (\ref{eq:KeyInt}) converges to, up to multiplying by a constant, 
\begin{equation}\notag
\int_\XX\langle g,\nabla(\phi \dd f)\rangle_{\sf HS} \,\dd\haus^n=\int_\XX\langle \dd\phi, \dd f \rangle \dd\haus^n+\int_\XX\phi \tr(\hess f)\dd\haus^n.
\end{equation}
This completes the proof.\end{proof}

To deduce from the above the equivalence of the `weak' and `strong' non-collapsed conditions we shall use the following simple result:

\begin{lemma}\label{prop:main} Let $(\X,\sfd,\mm)$ be an $\RCD(K,N)$ space. Also, let $U\subseteq\XX$ be an open connected set and let $\xi \in L^\infty_\loc(U, \meas)$. Assume that for every $\psi\in\Lip_{bs}(\X, \dist)$ with support in $U$ and $f\in D(\Delta)$ it holds 
\begin{equation}\label{intparthip}
\int_\XX \xi  \langle \nabla \psi,\nabla f\rangle \dd{\mass} =- \int_\XX\xi  \psi \Delta f \dd\mm.
\end{equation}
Then $\xi$ is constant on $U$.
\end{lemma}

\begin{proof}
It suffices to check that $\xi$ is locally constant on $U$ because $U$ is connected.
Let $z \in \XX$ and $r\in(0,\frac16)$ with $B_{3r}(z)\subset U$ and let $\psi\in\Lip(\X,\sfd)$ be  identically 1 on $B_{2r}(z)$ and with support in $B_{3r}(z)$.  Also, set  $\xi_t\defeq \heat_t \left( \chi_{B_{2r}(z)} \xi  \right)\in D(\Delta)$, namely $ \xi_t(y)=\int_{B_{2r}(z)} p(x,y,t) \xi(x) \dd\mm(x)$ for $\mm$-a.e.\ $y\in\X$ and notice  that Hille's theorem (see also Proposition \ref{prop:lochille}) gives 
\[
     \Delta \xi_t(y)=\int_{B_{2r}(z)} \Delta_yp(x,y,t)  \xi(x) \dd\mm(x)\stackrel{\eqref{eq:scambioLap}}=
          \int_{B_{2r}(z)} \xi\Delta p_{y,t} \dd{\mass}.
\]
This identity and the assumption \eqref{intparthip} (with $f=p_{y,t}$) give
\[
\Delta \xi_t(y)=\int_\X(\chi_{B_{2r}(z)}-\psi)\xi \Delta p_{y,t} \dd{\mass}-\int_\X\xi \langle\nabla\psi,\nabla p_{y,t}\rangle \,\d\mm
\]
for $\mm$-a.e.\ $y\in\X$. Therefore the assumption   $\xi \in L^{\infty}_{{loc}}(U, \meas)$ tells that  for $y\in B_{r}(z)$ we have
\[
\begin{split}
| \Delta \xi_t|(y) &\le C \int_{B_{3r}(z)\setminus B_{2r}(z)}|\Delta p_{y, t}|\dd{\mass}+C \int_{B_{3r}(z)\setminus B_{2r}(z)}|\nabla p_{y, t}|\dd\mass \\
\text{(by \eqref{eq:equi lip}, \eqref{eq:lapheatbd})}\qquad&\leq C \left(t^{-1}+t^{-1/2}\right) \exp \left(-\frac{ r^2}{5 t} \right)\int_{B_{3r}(z)}\frac1{\mm(B_{\sqrt t}(x))}\,\d\mm(x),
\end{split}
\]
where $C$ is a positive constant which is independent with $t$ and $y$. Now notice that \eqref{eq:bishgrom} and the assumption $r\in(0,\frac16)$ ensure that $\frac1{\mm(B_{\sqrt t}(x))}\leq \frac{C(K,N)}{\mm(B_1(z))}t^{-\frac N2}$ for every $t\in(0,1)$ and $x\in B_{3r}(z)$. It  then follows that $\Delta \xi_t$ uniformly converge to $0$ on $B_{r}(z)$. 

Let now $\phi\in\Lip(\X, \dist)$ be with support in $B_{r}(z)$ and notice that
\[
\begin{split}
\int_\XX|\d(\phi\xi_t)|^2\,\d\mm=\int_\XX |\xi_t|^2|\d\phi|^2+2\xi_t\phi\langle\d\xi_t,\d\phi\rangle+|\phi|^2|\d\xi_t|^2\,\d\mm=\int_\XX |\xi_t|^2|\d\phi|^2-|\phi|^2\xi_t\Delta\xi_t\,\d\mm.
\end{split}
\]
By what we proved we see that the RHS is bounded as $t\to 0^+$, hence the lower semicontinuity of the Cheeger energy ensures that  $\phi\xi\in H^{1,2}(\X,\sfd,\mm)$. Now choose $\phi\in \Lip(\X,\dist)$ identically 1 on $B_{r/2}(z)$ and with support in $B_r(z)$ and let $\eta\in \Lip(\X,\dist)$ be arbitrary with support in $B_{r/2}(z)$. Since $\supp(\eta)\subset\{\phi=1\}$, from \eqref{intparthip} it follows that 
\begin{equation}
\label{eq:withphi}
\int_\XX \phi \xi  \langle \nabla \eta,\nabla f\rangle \dd{\mass} =- \int_\XX\eta\xi  \phi \Delta f \dd\mm
\end{equation} 
for any  $f\in D(\Delta)$. Moreover, by what we just proved the following computations are justified:
\[
\begin{split}
- \int_\XX\phi\xi  \eta \Delta f \dd\mm&=\int_\XX \langle\nabla(\phi\xi  \eta) ,\nabla f\rangle \d\mm=\int_\XX \phi\xi \langle\nabla   \eta ,\nabla f\rangle + \eta\langle\nabla(\phi \xi)  ,\nabla f\rangle \d\mm.
\end{split}
\]
This and \eqref{eq:withphi} imply  that $\int_\XX  \eta\langle\nabla \xi  ,\nabla f\rangle \d\mm=\int_\XX  \eta\langle\nabla (\phi\xi)  ,\nabla f\rangle \d\mm=0$. The arbitrariness of $\eta$ then gives  $\langle\nabla( \phi\xi ) ,\nabla f\rangle =0$ $\mm$-a.e.\ on $B_{r/2}(z)$. Then the density of $D(\Delta)$ in $H^{1,2}(\X,\sfd,\mm)$ gives $\nabla(\phi \xi)=0$ $\mm$-a.e.\ on $B_{r/2}(z)$. In turn this implies (e.g.\ from the Sobolev to Lipschitz property) that $\phi\xi$, and thus $\xi$, has a representative which is constant in $B_{r/2}(z)$, which is sufficient to conclude.
\end{proof}
We have now all the ingredients to prove the main equivalence result of this manuscript.

\begin{proof}[Proof of Theorem \ref{thm:main}]
Under \eqref{main:uniformbound}, we can apply Theorem \ref{intbyparts} and deduce the integration-by-parts formula: 
\begin{equation}\notag
\int_\XX \langle \dd\phi,\dd f\rangle\dv{\haus^n}{\mass}\dd{\mass} =-\int_\XX \phi\tr(\hess f)\dv{\haus^n}{\mass}\dd{\mass},
\end{equation}
valid for any $\phi\in\Lip(\X, \dist)$ with support in $U$ and any $f\in D(\Delta)$. Now notice that \eqref{main:uniformbound} together with Theorem \ref{thm:RN}  imply that $\dv{\haus^n}{\mass} \in L^{\infty}_{{loc}}(U, \meas)$. Hence if item \ref{thm:item1} holds,  we can apply   Lemma \ref{prop:main} with  $\xi=\chi_U\dv{\haus^n}{\mass}$ to deduce that item \ref{thm:item2} holds as well.

Conversely, if item \ref{thm:item2} holds, for all $\phi$ and $f$ as above, we have
\begin{equation}\notag
-\int_\XX  \phi\Delta f\dd{\mass}=  \int_\XX \langle \dd\phi,\dd f\rangle\dd{\mass} =-\int_\XX \phi\tr(\hess f)\dd{\mass},
\end{equation}
having used item \ref{thm:item2} and the integration-by-parts formula in the last step. By the arbitrariness of $\phi$, this proves item \ref{thm:item1}.
\end{proof}

\begin{proof}[Proof of Theorem \ref{thm:mainres}] From the Bishop-Gromov inequality \eqref{eq:bishgrom} it easily follows that for any bounded set $A$ of $\XX$ we have

\begin{equation}\label{asdwe}
	\inf_{r \in (0, 1),x \in A}\frac{\meas(B_r(x))}{r^N}>0.
\end{equation}

On the other hand, Theorem \ref{thm:equivnon} gives that the essential dimension of $\X$ is $N$, thus Theorem \ref{thm:generalizedhan} with \eqref{eq:bochner}  shows \begin{equation}
	\label{eq:toconclude2}
	\Delta f=\tra(\hess f)\qquad\forall f\in D(\Delta).
\end{equation}
 Then the conclusion follows from \eqref{asdwe}, \eqref{eq:toconclude2} and Theorem \ref{thm:main}.
\end{proof}

\begin{proof}[Proof of Theorem \ref{thm:close}] From the continuity of $\haus^N$ in the compact (as a consequence of Theorem \ref{thm:moduli}) space of unit balls in $\RCD(K,N)$ spaces stated in Theorem \ref{thm:conthN}, we see that picking $\epsilon$ sufficiently small, the conclusion $\abs{\haus^N(B_1(x))-\haus^N(B_1(y))}<\delta$ holds true. Thus we concentrate on the first part of the claim.

The proof is done by contradiction. If not, there exist a sequence $\epsilon_i \to 0^+$, a sequence of pointed $\RCD(K, N)$ spaces $(\XX_i, \dist_i, \meas_i, x_i)$ and a sequence of non-collapsed $\RCD(K, N)$ spaces $(\YY_i, \dist_{\YY_i}, \haus^N, y_i)$ with $\haus^N(B_1(y_i))\ge v$ such that $(\XX_i, \dist_i,  x_i)$ $\epsilon_i$-pGH close to $(\YY_i, \dist_{\YY_i}, y_i)$ and so that $\meas_i$ is not proportional to $\haus^N$.

Thanks to Theorem \ref{thm:moduli}, after passing to a non-relabelled subsequence, there exists a pointed $\RCD(K, N)$ space $(\ZZ, \dist_\ZZ, \meas_\ZZ, z)$ such that 
\begin{equation}\notag
\left(\XX_i, \dist_i, \frac{1}{\meas_i(B_1(x_i))}\meas_i, x_i\right) \stackrel{\mathrm{pmGH}}{\to} (\ZZ, \dist_\ZZ, \meas_\ZZ, z)
\end{equation}
and 
\begin{equation}\notag
(\YY_i, \dist_{\YY_i}, y_i) \stackrel{\mathrm{pGH}}{\to} (\ZZ, \dist_\ZZ,  z).
\end{equation} 
Thanks to Theorem \ref{thm:pGH} with (\ref{eq:volumenot}), we have 
\begin{equation}\notag
(\YY_i, \dist_{\YY_i}, \haus^N, y_i) \stackrel{\mathrm{pmGH}}{\to} (\ZZ, \dist_\ZZ, \haus^N,  z),
\end{equation}
with $\haus^N(B_1(z)) \ge v$. Recalling Theorem \ref{thm:equivnon}, we see that $(\ZZ,\dist_\ZZ,\mass_\ZZ)$ is weakly non-collapsed, in particular, has essential dimension $N$. Then the lower semicontinuity statement given by Theorem \ref{thm:loweress} gives
\begin{equation}\notag
N \ge \liminf_{i \to \infty}\mathrm{essdim}(\XX_i) \ge \mathrm{essdim}(\ZZ)=N.
\end{equation}
It follows that $\mathrm{essdim}(\XX_i)=N$ for any sufficiently large $i$. Thus from the characterization of weakly non-collapsed spaces in Theorem \ref{thm:equivnon} and our main result Theorem \ref{thm:mainres} it follows that $\mm_i=c_i\haus^N$ for every $i$ sufficiently large. This provides the desired contradiction.
\end{proof}

\begin{proof}[Proof of Theorem \ref{thm:tangcha}]
Let us take $r_i \to 0^+$ with (\ref{eq:ulla_ulla}), according to the assumption $(\YY,\dist_\YY,\mass_\YY,y)\in\mathrm{Tan}(\XX, \dist,\meas,x)$. As the essential dimension does not change under rescaling as in the LHS of \eqref{eq:ulla_ulla}, we see, by Theorem \ref{thm:loweress} and the assumption $\mathrm{essdim}(\YY)=N$, that $\mathrm{essdim}(\XX)=N$.
Thus we conclude by our main result Theorem \ref{thm:mainres}, taking into account also Theorem \ref{thm:equivnon}.
\end{proof}

\begin{proof}[Proof of Theorem \ref{thm:fromfar}]
Let us take $r_i\rightarrow \infty$ and a sequence of rescaled spaces as in the LHS of \eqref{eq:ulla_ulla}; by Theorem \ref{thm:moduli} (here we use the fact that the space is an $\RCD(K,N)$ space with $K=0$) we can extract a non relabelled subsequence of $\{r_i\}_i$ such that such rescaled spaces converge to the $\RCD(0,N)$ space $(\YY,\dist_\YY,\mass_\YY,y)$ in the pmGH topology. Therefore, if $z\in\YY$, we take a sequence $\{y_i\}_i\subseteq\XX$ that converges to $z$ under this pmGH convergence,
$$
\frac{\mass_\YY(B_r(z))}{r^N}=\lim_i \frac{\mass(B_{r r_i}(y_i))}{r^N \mass(B_{r_i}(x))}=\lim_i \frac{\mass(B_{r r_i}(y_i))}{(r r_i)^N}\frac{r_i^N}{ \mass(B_{r_i}(x))}\le \limsup_i {\mass(B_1 (y_i))}\frac{r_i^N}{ \mass(B_{r_i}(x))}\le C
$$
where $C$ is independent of $r$. 
Here we have used the Bishop-Gromov inequality \eqref{eq:bishgrom} 
 for the first inequality and our assumptions for the last inequality. 
 Therefore, using Theorem \ref{thm:equivnon} we see that $\mathrm{essdim}(\YY)=N$, so that we can conclude as in the proof of Theorem \ref{thm:tangcha}.
\end{proof}

\bibliographystyle{siam}
\bibliography{../weakNC-NC.bib}
\end{document}